\pdfoutput=1
\RequirePackage{ifpdf}
\ifpdf % We~are running pdfTeX in pdf mode
\documentclass[pdftex]{sigma}
\else
\documentclass{sigma}
\fi

\usepackage{authblk}
\usepackage{blkarray}

\usepackage{mathrsfs,mathtools,bm}
\usepackage{chngcntr}

\usepackage{tikz}
\usetikzlibrary{calc}

\numberwithin{equation}{section}

\newcommand{\cG}{\mathcal{G}}

\newcommand{\vph}{\varphi}

\newcommand\atopn[2]{\genfrac{}{}{0pt}{}{#1}{#2}}

\newtheorem{theo}{Theorem}[section]
\newtheorem*{theo*}{Theorem}
\newtheorem{lemm}[theo]{Lemma}
\newtheorem{prop}[theo]{Proposition}

\newtheorem{coro}[theo]{Corollary}

\theoremstyle{definition}

\newtheorem{defi}[theo]{Definition}

\newtheorem{exam}[theo]{Example}
\newtheorem{rema}[theo]{Remark}

\newcommand{\rb}{\overline{r}}
\newcommand{\aw}{\mathfrak{aw}(n)}
\newcommand{\saw}{\mathfrak{saw}(n)}

\begin{document}
\allowdisplaybreaks

\newcommand{\arXivNumber}{2303.17677}

\renewcommand{\PaperNumber}{077}

\FirstPageHeading

\ArticleName{The Higher-Rank Askey--Wilson Algebra\\ and Its Braid Group Automorphisms}
\ShortArticleName{The Higher-Rank Askey--Wilson Algebra and Its Braid Group Automorphisms}

\Author{Nicolas CRAMP\'E~$^{\rm a}$, Luc FRAPPAT~$^{\rm b}$, Lo\"ic POULAIN D'ANDECY~$^{\rm c}$ and Eric RAGOUCY~$^{\rm b}$}
\AuthorNameForHeading{N.~Cramp\'e, L.~Frappat, L.~Poulain d'Andecy and E.~Ragoucy}

\Address{$^{\rm a)}$~Institut Denis-Poisson CNRS/UMR 7013 - Universit\'e de Tours - Universit\'e d'Orl\'eans,\\
\hphantom{$^{\rm a)}$}~Parc de Grandmont, 37200 Tours, France}
\EmailD{\href{mailto:crampe1977@gmail.com}{crampe1977@gmail.com}}

\Address{$^{\rm b)}$~Laboratoire d'Annecy-le-Vieux de Physique Th\'eorique LAPTh, Universit\'e Savoie Mont Blanc,\\
\hphantom{$^{\rm b)}$}~CNRS, F-74000 Annecy, France}
\EmailD{\href{mailto:luc.frappat@lapth.cnrs.fr}{luc.frappat@lapth.cnrs.fr}, \href{mailto:eric.ragoucy@lapth.cnrs.fr}{eric.ragoucy@lapth.cnrs.fr}}

\Address{$^{\rm c)}$~Laboratoire de math\'ematiques de Reims UMR 9008, Universit\'e de Reims Champagne-Ardenne,\\
\hphantom{$^{\rm c)}$}~Moulin de la Housse BP 1039, 51100 Reims, France}
\EmailD{\href{mailto:loic.poulain-dandecy@univ-reims.fr}{loic.poulain-dandecy@univ-reims.fr}}

\ArticleDates{Received April 12, 2023, in final form October 10, 2023; Published online October 18, 2023}

\Abstract{We propose a definition by generators and relations of the rank $n-2$ Askey--Wilson algebra $\aw$ for any integer $n$, generalising the known presentation for the usual case $n=3$. The generators are indexed by connected subsets of $\{1,\dots,n\}$ and the simple and rather small set of defining relations is directly inspired from the known case of $n=3$. Our first main result is to prove the existence of automorphisms of $\aw$ satisfying the relations of the braid group on $n+1$ strands. We also show the existence of coproduct maps relating the algebras for different values of~$n$. An immediate consequence of our approach is that the Askey--Wilson algebra defined here surjects onto the algebra generated by the intermediate Casimir elements in the $n$-fold tensor product of the quantum group ${\rm U}_q(\mathfrak{sl}_2)$ or, equivalently, onto the Kauffman bracket skein algebra of the $(n+1)$-punctured sphere. We also obtain a family of central elements of the Askey--Wilson algebras which are shown, as a~direct by-product of our construction, to be sent to $0$ in the realisation in the $n$-fold tensor product of ${\rm U}_q(\mathfrak{sl}_2)$, thereby producing a large number of relations for the algebra generated by the intermediate Casimir elements.}


\Classification{16T10; 33D45; 81R12}

\section{Introduction}

The (usual) Askey--Wilson algebra, denoted in this paper $\mathfrak{aw}(3)$, originally appeared in~\cite{Zh} to provide an algebraic
underpinning for the eponym polynomials. Indeed, these polynomials are solutions of a bispectral problem, i.e., they satisfy a recurrence and a difference relation.
By identifying the algebraic relations the recurrence operator and the difference operator obey, the relations of $\mathfrak{aw}(3)$ have been discovered.
Another occurrence of this algebra appears in the Racah problem which consists in studying the different recouplings of three irreducible representations of ${\rm U}_q(\mathfrak{sl}_2)$. This leads to an algebraic interpretation of the $6j$-symbols of ${\rm U}_q(\mathfrak{sl}_2)$, that appears to be intimately linked to the Askey–Wilson polynomials~\cite{Zh2,GZ}.
 This is based on the fact that the intermediate Casimir elements of ${\rm U}_q(\mathfrak{sl}_2)^{\otimes 3}$ verify the Askey–Wilson relations of $\mathfrak{aw}(3)$~\cite{Hua3}.
 This led to a new point of view on the centralisers of ${\rm U}_q(\mathfrak{sl}_2)$ in tensor products of any three possibly different spin representations~\cite{CPVZ, CVZ2}.
The Askey--Wilson algebras have subsequently found applications in different contexts as association schemes~\cite{BI,Leo,Ter}, Leonard pairs~\cite{TV}, Kauffman bracket skein algebras~\cite{CL}
and symmetry of physical models~\cite{GVZ, KKM,Post}. In addition of that, the representation theory has been studied in~\cite{Hua2}. For more details about the Askey--Wilson algebra $\mathfrak{aw}(3)$, we refer to the review~\cite{avatar}.

The incarnation of the Askey--Wilson algebra $\mathfrak{aw}(3)$ related to ${\rm U}_q(\mathfrak{sl}_2)^{\otimes 3}$ offers a natural path to a generalisation by considering ${\rm U}_q(\mathfrak{sl}_2)^{\otimes n}$ instead. This path has been already used in different contexts, for example identifying the intermediate Casimir operators as symmetries of quantum $q$-deformed Calogero--Gaudin superintegrable system~\cite{GIV}.
Similarly, the incarnation related to the
Kauffman bracket skein algebra of the sphere with $4$ punctures naturally suggests to increase the number of punctures to $n+1$. Fortunately, this was proved recently~\cite{CL} that the subalgebra of
${\rm U}_q(\mathfrak{sl}_2)^{\otimes n}$ generated by the intermediate Casimir elements and the Kauffman bracket skein algebra of the sphere with $n+1$ punctures
are isomorphic, so that these two ways of generalisation actually coincide. Following the terminology of~\cite{avatar}, we call the resulting algebra the ``special Askey--Wilson algebra'' and denote it $\saw$.

We emphasize that for $n=3$, the Askey--Wilson algebra $\mathfrak{aw}(3)$ and the special Askey--Wilson algebra $\mathfrak{saw}(3)$ are different.
The algebra $\mathfrak{aw}(3)$ has a simple definition in terms of $q$-commutation relations for 3 generators, resulting in an algebra with a polynomial PBW basis,
while the algebra $\mathfrak{saw}(3)$ is obtained by further quotienting out by a certain central element, which looks slightly complicated.
So this should come as no surprise that the problem of finding an explicit algebraic description of $\saw$, for any $n$, by generators and relations seems quite difficult.
We refer to the appendix in~\cite{CL} for the case $n=4$.

Due to the importance of $\mathfrak{aw}(3)$, different attempts to generalise its definition appeared previously for $n=4$~\cite{GW, PostWalt} or for any $n$~\cite{dBdC,dBdCvdV, DeC} but a complete set of relations had not been provided.
The point of view in the paper is that it would be interesting to have an Askey--Wilson algebra $\aw$,
which would be a genuine generalisation of $\mathfrak{aw}(3)$, and which would have the special algebra $\saw$ as a quotient.
This would provide the complete analogue for any $n$ of the general picture for $n=3$. This would also provide the quantum analogue of the $q=1$ classical case, where we have the higher-rank Racah algebras defined in terms of simple commutation relations, which admit as a (rather complicated) quotient the special Racah algebras describing the diagonal centraliser in ${\rm U}(\mathfrak{sl}_2)^{\otimes n}$, see~\cite{CGPV}. To be interesting, the algebra $\aw$ should have a rather simple and natural definition, and should enjoy natural properties. It could be then considered in particular as an intermediary step, interesting in its own as is $\mathfrak{aw}(3)$, towards the description of $\saw$.

In this paper, we provide a definition of the algebra $\aw$ in terms of generators and relations satisfying the above requirements, see Definition~\ref{prop:def-nico2}.
It possesses $\frac{n(n+1)}2$ generators, $n+1$ of them being central.
One property of the algebras $\aw$, which is our first main result, is the existence of coproduct maps relating $\aw$ to $\mathfrak{aw}(n+1)$,
and of a group of automorphisms, for each $\aw$, satisfying the relations of the braid group on $n+1$ strands.
The construction of the automorphisms is done in two steps. First, we construct maps satisfying the relations of the braid group on $n$ strands.
These maps mimic at the level of $\aw$ the natural coproduct maps and conjugation maps by the $R$-matrix in ${\rm U}_q(\mathfrak{sl}_2)^{\otimes n}$. They also have natural interpretations in the Kauffman bracket skein algebra. We indicate that an action of the braid group on $3$ strands by automorphisms of $\mathfrak{aw}(3)$ was obtained in~\cite{Ter}. The connection with the $R$-matrix of ${\rm U}_q(\mathfrak{sl}_2)$ was provided in~\cite{CGVZ}. We also show that, as for $n=3$ in~\cite{Ter}, the realisation of the braid group on $n$ strands as automorphisms actually factors through its quotient by the centre for any $n$. Second, we note that we have a larger group of automorphisms since we supplement the braid group on $n$ strands with another generator realising the braid group on $n+1$ strands. This additional generator has a natural interpretation in the skein algebra (since there are $n+1$ punctures) and seems to be new already for $n=3$.

We find it remarkable that the algebra $\aw$ with its relatively simple presentation is able to retain the properties of having coproduct maps and braid group automorphisms, which may seem to be intrinsic properties of the quotients $\saw$ realised in ${\rm U}_q(\mathfrak{sl}_2)^{\otimes n}$. In fact, our conceptual guiding principle for deciding how many and which relations to put in the definition of $\aw$ was the following: we put the minimal set of relations ensuring the existence of these coproduct maps and automorphisms (and of course such that $\textbf{aw}(3)$ is the known usual Askey--Wilson algebra). This approach via this kind of universal property turns out to be quite fruitful since, first, it results in a natural presentation of $\aw$ and second it allows us to collect easily several interesting consequences, as we discuss now.

First of all, the automorphisms and the coproduct maps allow to obtain directly many relations satisfied in $\aw$ as consequences of the defining relations.
We recover thus many relations calculated for example in~\cite{CL, DeC, PostWalt}. Secondly, this approach allows us to obtain without any calculation that the algebra $\aw$ indeed surjects onto the algebra in ${\rm U}_q(\mathfrak{sl}_2)^{\otimes n}$
generated by the intermediate Casimir elements, or equivalently onto the Kauffman bracket skein algebra of the $(n+1)$-punctured sphere.

Then in a second part of the paper, we turn to more involved consequences of the defining relations of $\aw$, namely, the existence of a large family of central elements.
Here also the coproduct maps and the automorphisms are put to full use. Indeed, we start with the known central element of $\textbf{aw}(3)$ \big(the one which is sent to $0$ in the special quotient realized in ${\rm U}_q(\mathfrak{sl}_2)^{\otimes 3}$\big),
and we obtain a family of central elements in $\aw$ by applying as much as we can the coproduct maps and the automorphisms. We are able to describe a minimal spanning set for this family of central elements,
indexed by subsets of $\{1,\dots,n\}$. On this family of central elements, the action of the $n$-strands braid group automorphisms is shown to be realised simply via the permutation action of the symmetric group on $n$ letters.
Again, as an immediate consequence of the approach advocated here, we have that all these central elements become $0$ in ${\rm U}_q(\mathfrak{sl}_2)^{\otimes n}$.
Therefore, we get many relations satisfied by the quotient $\saw$. However, it remains an open question to decide if putting all these central elements to 0 is enough to get a presentation of $\saw$.
It is actually an open question whether $\saw$ is a quotient of $\aw$ by a central ideal. Comparing with the appendix of~\cite{CL}, we were able to check this latter property only for $n=4$, but we still do not know whether our family of central elements generate the ideal.

Other open questions also remain mainly about the uses of the higher rank Askey--Wilson algebra. For example, a quotient of $\aw$ should describe the centralisers of ${\rm U}_q(\mathfrak{sl}_2)$
in tensor products of any $n$ possibly different spin representations, generalising the results of~\cite{CPVZ, CVZ2} obtained for $n=3$.
Further, the connection with the multivariate Askey--Wilson polynomials~\cite{GI2,G, Ili} should be very fruitful. Indeed, the
study of the representations of the algebra $\aw$ provides the bispectrality operators of these polynomials,
generalising to the $q$-deformed case the known connections between the multivariate Racah polynomials and the higher-rank Racah algebras~\cite{CFR}. Finally, we performed some computer aided proofs, see for instance Proposition~\ref{propcas4}: a conceptual proof of it would be nice to achieve.

{\bf Organisation.} We give the notations and the definition of the algebra $\aw$ in Section~\ref{sec-def} and we prove many consequences of the defining relations.
The description of the group of automorphisms realising the braid group is in Section~\ref{sec-auto} and the technical part of the proof is postponed to Appendix~\ref{sec-appendix}.
Section~\ref{sec-coproduct} deals with the coproduct maps and their relations with the automorphisms of the previous section.
The central elements of $\aw$ and their properties are given in Section~\ref{sec-Cas}. The connections with ${\rm U}_q(\mathfrak{sl}_2)^{\otimes n}$ and with the skein algebra are given in Section~\ref{sec-connections} while
the links with the higher rank Racah algebra are found in Section~\ref{sec-limit}.

%\setcounter{tocdepth}{1}
%\tableofcontents

\section[The Askey--Wilson algebra aw(n)]{The Askey--Wilson algebra $\boldsymbol{\aw}$}\label{sec-def}

In this section, we give a definition of $\aw$ by generators and relations and draw some consequences of this definition.
The algebra $\aw$ considered here is over $\mathbb{C}(q)$ for an indeterminate~$q$. It is also defined over $\mathbb{C}$ if
we take $q$ a non-zero complex number such that $q^2\neq 1$.

\subsection{Definition}

Let us define some notations and terminology on sets and subsets:
\begin{itemize}\itemsep=0pt
 \item For two subsets $I,J\subseteq\{1,\dots,n\}$, we say that $I<J$ if all elements of $I$ are strictly smaller than all elements of $J$.

\item A non-empty subset $I\subseteq\{1,\dots,n\}$ is connected if it consists in a subset of consecutive integers.

\item Two disjoint connected subsets $I,J\subseteq\{1,\dots,n\}$ are adjacent if their union is connected.

\item A hole $H$ between two disjoint connected subsets $I_1$ and $I_2$ consists in the connected subset between $I_1$ and $I_2$. Visually, we have
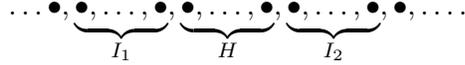
\[\ \ \ldots \bullet,\underbrace{\bullet,\dots,\bullet}_{I_1},\underbrace{\bullet,\dots,\bullet}_{H},\underbrace{\bullet,\dots,\bullet}_{I_2},\bullet,\ldots .\]
In this picture, and in all the similar pictures below, the integers (drawn as $\bullet$'s) are ordered either from left to right or from right to left (depending on the respective positions of $I_1$ and $I_2$ in the natural order). So such a picture does not mean that $I_1<I_2$, it means that~$I_1$,~$H$,~$I_2$ are adjacent connected subsets.

\item
A sequence $(I_1,\dots,I_k)$ of non-empty connected subsets of $\{1,\dots,n\}$ is said monotonic if
either $I_1<I_2<\dots<I_k$ or $I_1>I_2>\dots>I_k$.
 \end{itemize}

By direct computation, we can show the following relations, called $q$-Jacobi relations:
\begin{subequations}
\begin{align}
 &\bigl[[A,B]_q,C\bigr]_q-\bigl[A,[B,C]_q\bigr]_q=\frac{1}{\big(q-q^{-1}\big)^2}\bigl[B,[C,A]\bigr] , \label{eq:qjacobi1}\\
 &\bigl[A,[C,B]_q\bigr]_q-\bigl[C,[A,B]_q\bigr]_q=\frac{q+q^{-1}}{q-q^{-1}}\bigl[[A,C],B\bigr]_{q^2} ,\label{eq:qjacobi2}\\
 &\bigl[A,[B,C]_q\bigr]+\bigl[C,[A,B]_q\bigr]+\bigl[B,[C,A]_{q}\bigr]=0 ,\label{eq:qjacobi3}
\end{align}
\end{subequations}
where the $q$-commutator is defined by
\begin{equation} \label{eq:qcom}
 [A,B]_q=\frac{1}{q-q^{-1}}\big(qAB- q^{-1}BA\big)=[B,A]_{q^{-1}} ,
\end{equation}
and $[A,B]=AB-BA$ is the usual commutator.
Of particular interest is the case when~$A$ and~$C$ commute. In this case, the right-hand sides of~\eqref{eq:qjacobi1} and~\eqref{eq:qjacobi2} vanish
and we have
\begin{align*}
[A,C]_q=AC,\qquad [A,BC]_q=[A,B]_qC,\qquad [A,CB]_q=C[A,B]_q
 \end{align*}
that we will use without mentioning.

In the following, we consider the elements $C_I$ indexed by all the connected subsets $I\subseteq\{1,\dots,n\}$.
By convention, we set
\begin{equation}\label{eq:Cvide}
 C_{\varnothing}:=1 .
\end{equation}
A notation for $C_I$ forgetting the accolades for a set will be used: if $I=\{i,i+1,\dots,j\}$, $C_{ii+1\dots j}$ stands for $C_I$.
\begin{exam}
For $n=3$, the elements $C_I$ are $C_1$, $C_2$, $C_3$, $C_{12}$, $C_{23}$, $C_{123}$ and, for $n=4$, they are $C_1$, $C_2$, $C_3$, $C_4$, $C_{12}$, $C_{23}$, $C_{34}$, $C_{123}$, $C_{234}$, $C_{1234}$.
\end{exam}
 If a connected subset $I$ is written as the disjoint union of two connected subsets $I_1$ and $I_2$, then~$C_{I_1I_2}$ and $C_{I_2I_1}$ mean $C_I$.
\begin{exam}
If $I_1=\{1\}$ and $I_2=\{2,3\}$, the notations $C_{I_1I_2}$ and $C_{I_2I_1}$ both mean $C_{123}$.
\end{exam}

We also want to define elements $C_{I_1I_2}$ with a hole between $I_1$ and $I_2$.
Let us consider~$I_1$,~$I_2$ two non-empty and disjoint connected subsets of $\{1,\dots,n\}$ with a non-empty hole $H$ between them. Visually, we have
\[\ \ \ldots \bullet,\underbrace{\bullet,\dots,\bullet}_{I_1},\underbrace{\bullet,\dots,\bullet}_{H},\underbrace{\bullet,\dots,\bullet}_{I_2},\bullet,\ldots, \]
where the integers $1,\dots,n$, (drawn as $\bullet$'s) are ordered either from left to right or from right to left
(depending on the respective positions of $I_1$ and $I_2$ in the natural order). The element $C_{I_1I_2}$ is defined by
\begin{equation}\label{def-CI2}
%\tag{\textbf{def}}
C_{I_1I_2}:=-[C_{I_1H},C_{HI_2}]_q+C_{I_1}C_{I_2}+C_{H}C_{I_1HI_2}.
\end{equation}
Sometimes a ``,'' in $C_{I_1,I_2}$ is inserted for clarity.
\begin{exam}%\label{ex:C13}
As an illustration, we have
\begin{gather*}
 C_{13} :=-[C_{12},C_{23}]_q+C_1C_3+C_2C_{123} ,\qquad
 C_{236} :=-[C_{2345},C_{456}]_q+C_{23}C_6+C_{45}C_{23456} .
\end{gather*}
For $n=4$, the relation~\eqref{def-CI2} defines the elements $C_{13}$, $C_{31}$, $C_{24}$, $C_{42}$, $C_{14}$, $C_{41}$, $C_{12,4}$, $C_{4,12}$, $C_{1,34}$, $C_{34,1}$.
\end{exam}
We are ready to give a definition of the algebra $\aw$ by generators and relations, using the notations introduced above.
\begin{defi}\label{prop:def-nico2}
The algebra $\aw$ is the unital associative algebra generated by the ele\-ments~$C_I$, where $I$ is any non-empty connected subset of $\{1,\dots,n\}$,
satisfying the following relations:
\begin{itemize}\itemsep=0pt
\item for any two connected subsets $I$ and $J$,
\begin{equation}\label{relcommv}
 [C_I,C_J]=0\qquad \text{if $I\cap J=\varnothing$ or $I\subset J$;}
\end{equation}

\item for any monotonic sequence of three adjacent non-empty connected subsets $(I_1,I_2,I_3)$,
\begin{align}
 & C_{I_1I_2}=-[C_{I_2I_3},C_{I_1I_3}]_q+C_{I_1}C_{I_2}+C_{I_3}C_{I_1I_2I_3} , \label{relaw31v}
\end{align}
where $C_{I_1I_3}$ is defined by~\eqref{def-CI2};

\item for any monotonic sequence of four adjacent non-empty connected subsets $(I_1,I_2,I_3,I_4)$,
\begin{align}
 & C_{I_1I_4}=-[C_{I_1I_3},C_{I_3I_4}]_q+C_{I_1}C_{I_4}+C_{I_3}C_{I_1I_3I_4} ,\label{relaw41v}
\end{align}
where $C_{I_1I_3}$, $C_{I_1I_4}$ and $C_{I_1I_3I_4}$ are defined by~\eqref{def-CI2}.
\end{itemize}
\end{defi}

Note that in $\aw$, the element $C_{I_2I_1}$, obtained from~\eqref{def-CI2}, is in general different from~$C_{I_1I_2}$. However, it
is obtained by simply replacing $q$ by $q^{-1}$ in $C_{I_1I_2}$:
\begin{equation*}%\label{def-CI22}
C_{I_2I_1}:=-[C_{I_1H},C_{HI_2}]_{q^{-1}}+{C_{I_1}C_{I_2}}+C_{H}C_{I_1HI_2}.
\end{equation*}
We have used that, from our notations, we get $C_{I_1H}=C_{HI_1}$ and other similar equalities implying adjacent subsets.

\begin{lemm} Let $(I_1,H,I_2)$ be any sequence of monotonic adjacent connected subsets. In $\aw$, the elements $C_{I_1I_2}$ and $C_{I_2I_1}$ are connected as follows:
\begin{subequations}\label{eq:C13C31}
\begin{align}
&\frac{q^{-1}}{q+q^{-1}}C_{I_1I_2} +\frac{q}{q+q^{-1}}C_{I_2I_1}= -C_{HI_2}C_{I_1H}+ C_{I_1}C_{I_2}+C_HC_{I_1HI_2},\\
& \frac{q}{q+q^{-1}}C_{I_1I_2}+\frac{q^{-1}}{q+q^{-1}}C_{I_2I_1}= -C_{I_1H}C_{HI_2}+ C_{I_1}C_{I_2}+C_HC_{I_1HI_2}.
\end{align}
\end{subequations}
\end{lemm}
\begin{proof} Replace $C_{I_1I_2}$ and $C_{I_2I_1}$ by their definition to get the results. \end{proof}

\begin{exam}\label{ex:aw3}
For $n=3$, the Askey--Wilson algebra $\mathfrak{aw}(3)$ is generated by $C_1$, $C_2$, $C_3$, $C_{12}$, $C_{23}$ and $C_{123}$.
Relation~\eqref{relcommv} proves that $C_1$, $C_2$, $C_3$ and $C_{123}$ are central. The defining relations~\eqref{def-CI2} lead to
\begin{align}\label{eq:aw31}
&C_{13} := -[C_{12},C_{23}]_q+ C_{1}C_{3}+C_{2}C_{123} .
\end{align}
The relations~\eqref{relaw41v} do not exist for $n=3$. The ones given by~\eqref{relaw31v} with the subsets~$(1,2,3)$ or~$(3,2,1)$ read
\begin{align}
&C_{12} = -[C_{23},C_{13}]_q+C_{1}C_{2}+C_{3}C_{123} ,\label{eq:aw32}
\\
&C_{23} = -[C_{12},C_{31}]_q+ C_{2}C_{3}+C_{1}C_{123} .\label{eq:aw33}
\end{align}
Using the definition~\eqref{def-CI2} of $C_{31}$ and the $q$-Jacobi relation~\eqref{eq:qjacobi1}, one proves that~\eqref{eq:aw33} can be replaced by
\begin{equation}
 C_{23} = -[C_{13},C_{12}]_q+ C_{2}C_{3}+C_{1}C_{123} .\label{eq:aw33b}
\end{equation}
Relations~\eqref{eq:aw31},~\eqref{eq:aw32} and~\eqref{eq:aw33b} are the usual defining relations of the Askey--Wilson algebra~$\mathfrak{aw}(3)$~\cite{Zh} (see also~\cite{avatar}).
Let us remark that there may be a change of normalisation for the generators used here and the ones used in the previous literature.
For example, there is the change
\[
C_I \to \frac{C_I}{q+q^{-1}},
\]
to compare with~\cite{avatar, DeC} or $C_I \to -C_I$ to compare with~\cite{PostWalt}.
The normalisation of this paper is chosen so that we have~\eqref{eq:Cvide}: $C_{\varnothing}=1$. Let us also point out that there is an unusual denominator in the definition~\eqref{eq:qcom} of the $q$-commutator.
\end{exam}

\subsection{Properties}%\label{sect:add-prop}

It is easily shown from~\eqref{relcommv} that
\begin{equation*}
 C_i, \ i=1,2,\dots,n, \quad C_{12\dots n} \text{ are central in } \aw.
\end{equation*}

There are also other relations implied by Definition~\ref{prop:def-nico2}.
\begin{lemm}\label{prop:def-nico}
In the algebra $\aw$, the following relations are also satisfied:
\begin{itemize}\itemsep=0pt
\item for any monotonic sequence of four adjacent non-empty connected subsets $(I_1,I_2,I_3,I_4)$, the following commutations hold:
\begin{subequations}\label{rel:com}
\begin{alignat}{3}
 & [C_{I_1I_2}, C_{I_1I_2I_4} ]=0 ,\qquad&& [C_{I_1I_2}, C_{I_4I_2I_1} ]=0 ,&\label{rel:com1}\\
 & [C_{I_3I_4}, C_{I_1I_3I_4} ]=0 ,\qquad&& [C_{I_3I_4}, C_{I_4I_3I_1} ]=0 ,&\nonumber\\ %\label{rel:com2}
 & [C_{I_1I_2I_3}, C_{I_1I_3}]=0 ,\qquad&& [C_{I_1I_2I_3}, C_{I_3I_1}]=0 ,&\nonumber\\ %\label{rel:com3}
 & [C_{I_2I_3I_4}, C_{I_2I_4}]=0 ,\qquad&& [C_{I_2I_3I_4}, C_{I_4I_2}]=0 ,&\nonumber\\ %\label{rel:com4}
 & [C_{I_2I_3}, C_{I_1I_4} ]=0 ,\qquad&& [C_{I_2I_3}, C_{I_4I_1} ]=0 ;&\nonumber %\label{rel:com5}
\end{alignat}
\end{subequations}

\item for any monotonic sequence of three adjacent non-empty connected subsets $(I_1,I_2,I_3)$,
\begin{subequations}\label{relaw3}
\begin{align}
 & C_{I_1I_2}=-[C_{I_2I_3},C_{I_1I_3}]_q+C_{I_1}C_{I_2}+C_{I_3}C_{I_1I_2I_3} , \label{relaw31}\\
 & C_{I_2I_3}=-[C_{I_1I_3},C_{I_1I_2}]_q+C_{I_2}C_{I_3}+C_{I_1}C_{I_1I_2I_3}, \label{relaw32}\\
&C_{I_1I_3}=-[C_{I_1I_2},C_{I_2I_3}]_q+C_{I_1}C_{I_3}+C_{I_2}C_{I_1I_2I_3}\label{relaw33};
\end{align}
\end{subequations}

\item for any monotonic sequence of four adjacent non-empty connected subsets $(I_1,I_2,I_3,I_4)$,%
\begin{subequations}\label{relaw4a}
\begin{align}
 & C_{I_1I_4}=-[C_{I_1I_2},C_{I_2I_4}]_q+ C_{I_1}C_{I_4}+C_{I_2}C_{I_1I_2I_4} ,\label{relaw43}\\
 & C_{I_2I_4}=-[C_{I_1I_4},C_{I_1I_2}]_q+ C_{I_2}C_{I_4}+C_{I_1}C_{I_1I_2I_4} ,\label{relaw46}\\
&C_{I_1I_2}=-[C_{I_2I_4},C_{I_1I_4}]_q+C_{I_1}C_{I_2}+C_{I_4}C_{I_1I_2I_4},\label{eq:2h3}
\end{align}\vspace{-1.1cm}
\end{subequations}
\begin{subequations}\label{relaw4b}
\begin{align}
 & C_{I_1I_4}=-[C_{I_1I_3},C_{I_3I_4}]_q+C_{I_1}C_{I_4}+C_{I_3}C_{I_1I_3I_4} ,\label{relaw41}\\
 & C_{I_1I_3}=-[C_{I_3I_4},C_{I_1I_4}]_q+ C_{I_1}C_{I_3}+C_{I_4}C_{I_1I_3I_4} ,\label{relaw45}\\
&C_{I_3I_4}=-[C_{I_1I_4},C_{I_1I_3}]_q+C_{I_3}C_{I_4}+C_{I_1}C_{I_1I_3I_4},\label{eq:2h2}
\end{align}\vspace{-1.1cm}
\end{subequations}
\begin{subequations}\label{relaw4c}
\begin{align}
 & C_{I_1I_2I_4}=-[C_{I_3I_1},C_{I_2I_3I_4}]_q+C_{I_1}C_{I_2I_4}+C_{I_3}C_{I_1I_2I_3I_4} ,\label{relaw47}\\
 &C_{I_3I_1}=-[C_{I_2I_3I_4},C_{I_1I_2I_4}]_q+C_{I_3}C_{I_1}+C_{I_2I_4}C_{I_1I_2I_3I_4} ,\label{relaw49}\\
&C_{I_2I_3I_4}=-[C_{I_1I_2I_4}, C_{I_3I_1}]_q +C_{I_3}C_{I_2I_4}+C_{I_1}C_{I_1I_2I_3I_4},\label{eq:2h4}
\end{align}\vspace{-1.1cm}
\end{subequations}
\begin{subequations}\label{relaw4d}
\begin{align}
 & C_{I_1I_3I_4}=-[C_{I_1I_2I_3},C_{I_4I_2}]_q+C_{I_4}C_{I_1I_3}+C_{I_2}C_{I_1I_2I_3I_4} , \label{relaw48} \\
 &C_{I_4I_2}=-[C_{I_1I_3I_4},C_{I_1I_2I_3}]_q+C_{I_2}C_{I_4}+C_{I_1I_3}C_{I_1I_2I_3I_4} ,\label{relaw410}\\
&C_{I_1I_2I_3}=-[C_{I_4I_2}, C_{I_1I_3I_4}]_q +C_{I_2}C_{I_1I_3}+C_{I_4}C_{I_1I_2I_3I_4},\label{eq:2h5}
\end{align}\vspace{-1.1cm}
\end{subequations}
\begin{subequations}\label{relaw4e}
\begin{align}
 & C_{I_1I_2I_4}=-[C_{I_2I_3},C_{I_1I_3I_4}]_q +C_{I_2}C_{I_1I_4}+C_{I_3}C_{I_1I_2I_3I_4} , \label{relaw42}\\
 & C_{I_1I_3I_4}=-[C_{I_1I_2I_4},C_{I_2I_3}]_q+ C_{I_1I_4}C_{I_3}+C_{I_2}C_{I_1I_2I_3I_4} ,\label{relaw44}\\
&C_{I_2I_3}=-[C_{I_1I_3I_4},C_{I_1I_2I_4}]_q+C_{I_2}C_{I_3}+C_{I_1I_4}C_{I_1I_2I_3I_4}. \label{eq:2h1}
\end{align}
\end{subequations}
\end{itemize}
\end{lemm}
\begin{proof} To prove the relations in~\eqref{rel:com}, replace the elements with one hole using~\eqref{def-CI2} and remark that all the generators appearing commute with the term without any hole using~\eqref{relcommv}.

Relation~\eqref{relaw31} is the relation~\eqref{relaw31v} in the Definition~\ref{prop:def-nico2}.
To prove relation~\eqref{relaw32}, we follow the same steps as in Example~\ref{ex:aw3}. Relation~\eqref{relaw33} is just~\eqref{def-CI2}.

Relation~\eqref{relaw41} is just a copy of~\eqref{relaw41v}.
To prove~\eqref{relaw43}, we evaluate the $q$-commutator by replacing $C_{I_2I_4}$ by its defining formula~\eqref{def-CI2}. Then we use the $q$-Jacobi relations~\eqref{eq:qjacobi1} and the relation~\eqref{relcommv}.

The other relations are proven similarly:~\eqref{relaw45} is obtained expressing $C_{I_1I_4}$ by~\eqref{relaw43} in the $q$-commutator $[C_{I_3I_4},C_{I_1I_4}]_q$, while
\eqref{relaw46} is deduced from~\eqref{relaw41v} and~\eqref{relaw47} is deduced from~\eqref{relaw46} (with $(I_1,I_2,I_3,I_4)\! \to \! (I_4,I_3,I_2,I_1)$), and~\eqref{relaw48} from~\eqref{relaw45} (with $(I_1,I_2,I_3,I_4)\!\to (I_4,I_3,I_2,I_1)$).
Relations~\eqref{relaw49},~\eqref{relaw410},~\eqref{relaw42} and~\eqref{relaw44} are deduced from~\eqref{def-CI2}.

To prove~\eqref{eq:2h1} (resp.~\eqref{eq:2h2},~\eqref{eq:2h3},~\eqref{eq:2h4},~\eqref{eq:2h5}), we use the definition~\eqref{def-CI2} of $C_{I_1I_3I_4}$ (resp. $C_{I_1I_3}$, $C_{I_2I_4}$, $C_{I_3I_1}$, $C_{I_4I_2}$),
the $q$-Jacobi relation and the relations just proven.
\end{proof}

\begin{rema}
In Lemma~\ref{prop:def-nico}, the relations are organized in clusters of three. These clusters correspond to three equations defining a $\mathfrak{aw}(3)$ algebra.
\end{rema}

\begin{lemm}\label{lem:relc1}
 In $\aw$, the following relations between commutators hold, for any monotonic sequence of four adjacent non-empty connected subsets $(I_1,I_2,I_3,I_4)$,
 \begin{align}
& [C_{I_1I_2},C_{I_2I_3} ]= [C_{I_1I_2I_4},C_{I_2I_3I_4} ]+ [C_{I_3I_4},C_{I_1I_4} ] ,
\label{eq:comut1}\\
& [C_{I_2I_3},C_{I_3I_4} ]= [C_{I_1I_2I_3},C_{I_1I_3I_4} ]+ [C_{I_1I_4},C_{I_1I_2} ] ,
\label{eq:comut2}\\
& [C_{I_1I_2} ,C_{I_2I_3I_4}]= [C_{I_1I_2I_4},C_{I_2I_3} ]+ [C_{I_1I_2I_3} ,C_{I_2I_4}] ,
\label{eq:comut3}\\
& [C_{I_3I_4} ,C_{I_1I_2I_3}]= [C_{I_1I_3I_4},C_{I_2I_3} ]+ [C_{I_2I_3I_4} ,C_{I_1I_3}] ,
\label{eq:comut4}\\
& [C_{I_1I_2I_3} ,C_{I_2I_3I_4}]= [C_{I_1I_2},C_{I_2I_4} ]+ [C_{I_3I_1} ,C_{I_3I_4}] ,
\label{eq:comut5}
 \end{align}
 and
 \begin{align}
& [C_{I_1I_3},C_{I_1I_2I_4} ]= [C_{I_1I_3I_4},C_{I_1I_2} ]+ [C_{I_1I_2I_3},C_{I_1I_4} ] ,
\label{eq:comut6}\\
& [C_{I_2I_4},C_{I_1I_3I_4} ]= [C_{I_1I_2I_4},C_{I_3I_4} ]+ [C_{I_2I_3I_4},C_{I_1I_4} ] ,
\label{eq:comut7}\\
& [C_{I_1I_4},C_{I_3I_1} ]= [C_{I_2I_3},C_{I_2I_4} ]+ [C_{I_1I_2I_4},C_{I_1I_2I_3} ] .
\label{eq:comut8}
 \end{align}
\end{lemm}
\begin{proof}
We first prove relation~\eqref{eq:comut1}. Relations~\eqref{eq:comut2}--\eqref{eq:comut5} follow the same steps.
We use the relation
\begin{equation}\label{eq:com-qcom}
[A,B] = \frac{q-q^{-1}}{q+q^{-1}} \big( [A , B]_q -[B , A]_q\big) ,
\end{equation}
to split each commutator in~\eqref{eq:comut1}, and simplify the global factor $\frac{q-q^{-1}}{q+q^{-1}}$.
Then, the left-hand side of~\eqref{eq:comut1} leads to $-C_{I_1I_3}+C_{I_3I_1}$. For the right-hand side, using~\eqref{relaw49} and~\eqref{relaw45}, we get
\begin{gather*}
{\rm r.h.s.}= [C_{I_1I_2I_4},C_{I_2I_3I_4} ]_q +C_{I_3I_1} -C_{I_1}C_{I_3}- C_{I_2I_4}C_{I_1I_2I_3I_4}
 -C_{I_1I_3}+C_{I_1}C_{I_3} \\
\hphantom{{\rm r.h.s.}=}{}+C_{I_4}C_{I_1I_3I_4}-[C_{I_1I_4} ,C_{I_3I_4}]_q.
\end{gather*}
Now using~\eqref{def-CI2} for $C_{I_1I_2I_4}$, the $q$-Jacobi identity, one obtains
 \begin{gather*}
{\rm r.h.s.}= [-[C_{I_1I_2I_3},C_{I_2I_3I_4} ]_q,C_{I_3I_4} ]_q +C_{I_4}[C_{I_1I_2},C_{I_2I_3I_4}]_q+C_{I_3}C_{I_1I_2I_3I_4}C_{I_2I_3I_4}+C_{I_3I_1}\\
\hphantom{{\rm r.h.s.}=}{} - C_{I_2I_4}C_{I_1I_2I_3I_4}
 -C_{I_1I_3}+C_{I_4}C_{I_1I_3I_4}
-[C_{I_1I_4} ,C_{I_3I_4}]_q.
 \end{gather*}
Finally, using~\eqref{def-CI2} again for $[C_{I_1I_2I_3},C_{I_2I_3I_4} ]_q$ and $[C_{I_1I_2},C_{I_2I_3I_4}]_q$, we obtain ${\rm r.h.s.}=C_{I_3I_1} -C_{I_1I_3}$, which proves~\eqref{eq:comut1}.

Using~\eqref{eq:com-qcom} and~\eqref{relaw31v}, relation~\eqref{eq:comut6} is equivalent to
\begin{gather*}
[C_{I_1I_3}, C_{I_1I_2I_4} ]_q - [C_{I_1I_2I_4},C_{I_1I_3} ]_q\\
\qquad= C_{I_2}C_{I_3I_4} - C_{I_4}C_{I_2I_3}
-[C_{I_1I_2},C_{I_1I_3I_4}]_q+[C_{I_1I_2I_3},C_{I_1I_4}]_q.
 \end{gather*}
Then, we replace the first $C_{I_1I_2I_4}$ and the second $C_{I_1I_3}$ by their definition~\eqref{def-CI2},
$C_{I_1I_3I_4}$ by~\eqref{relaw44} and $C_{I_1I_4}$ by~\eqref{relaw41v}. After some manipulations, we get
\begin{gather*}
C_{I_3}( [C_{I_1I_2},C_{I_1I_4} ]_q -[C_{I_1I_2I_3},C_{I_1I_3I_4}]_q
+ C_{I_1I_3}C_{I_1I_2I_3I_4} - C_{I_1}C_{I_1I_2I_4})=0.
 \end{gather*}
This last relation is proven replacing $C_{I_1I_4}$ and $C_{I_1I_3I_4}$ by their definition~\eqref{def-CI2} that proves~\eqref{eq:comut6}.
Relations~\eqref{eq:comut7} and \eqref{eq:comut8} are proven similarly.
\end{proof}

In the following lemma, non-trivial commuting relations between elements of $\aw$ are provided.
\begin{lemm}\label{lemm:relcom2}
In the algebra $\aw$, the following relations are also satisfied for any monotonic sequence of four adjacent non-empty connected subsets $(I_1,I_2,I_3,I_4)$,
\begin{subequations}%\label{rel:coma}
\begin{alignat}{3}
 & [C_{I_1I_3}, C_{I_4I_2} ]=0 ,&&& \label{rel:coma1}\\
 & [C_{I_1I_3I_4}, C_{I_1I_3}]=0 , \qquad&& [C_{I_1I_2I_4}, C_{I_2I_4}]=0 , &\label{rel:coma3}\\
 & [C_{I_1I_3I_4}, C_{I_1I_4}]=0 , \qquad&& [C_{I_1I_2I_4}, C_{I_1I_4}]=0.&\label{rel:coma4}
\end{alignat}
\end{subequations}
\end{lemm}

\begin{proof}
To prove~\eqref{rel:coma1}, we use~\eqref{def-CI2} to express $C_{I_4I_2}$ to get
\begin{align*}
[C_{I_1I_3}, C_{I_4I_2} ]&= [C_{I_1I_3}, -[C_{I_3I_4},C_{I_2I_3}]_q +C_{I_3}C_{I_2I_3I_4}]
\\
&= [C_{I_3I_4}, [C_{I_2I_3},C_{I_1I_3}]_q] + [C_{I_2I_3}, [C_{I_1I_3},C_{I_3I_4}]_q] +C_{I_3}[C_{I_1I_3},C_{I_2I_3I_4}],
\end{align*}
where we used the $q$-Jacobi identity~\eqref{eq:qjacobi3} in the second step. Then, the above $q$-commutators are expressed using~\eqref{relaw31v} and~\eqref{relaw41v}.
Finally,~\eqref{eq:comut4} (multiplied by $C_{I_3}$) ends the proof of~\eqref{rel:coma1}.

The expression of $C_{I_1I_3I_4}$ (resp. of $C_{I_1I_2I_4}$) given by~\eqref{relaw48} (resp.\ by~\eqref{relaw47}) now contains only elements commuting with $C_{I_1I_3}$ (resp.\ $C_{I_2I_4}$)
which shows that relations~\eqref{rel:coma3} hold. Relations~\eqref{rel:coma4} are proven similarly using relations~\eqref{relaw43} and~\eqref{relaw41v}.
\end{proof}

A relation $[C_{I_1I_3}, C_{I_4I_2} ]=0$, as proven in the lemma, appeared already in the realisation studied in~\cite{PostWalt}.
Let us also emphasize that $[C_{I_1I_3}, C_{I_2I_4} ], [C_{I_1I_3}, C_{I_3I_1} ]\neq 0$. Nevertheless, these commutators have nice expressions.
\begin{lemm}\label{lem:ff}For any monotonic sequence of four adjacent non-empty connected subsets $(I_1,I_2,\allowbreak I_3,I_4)$, one gets
\begin{gather}
\frac{[C_{I_1I_3}, C_{I_3I_1} ]}{q^2-q^{-2}}= C_{I_2I_3}^2 -C_{I_1I_2}^2 - C_{I_2I_3}(C_{I_2}C_{I_3}+C_{I_1}C_{I_1I_2I_3}) \nonumber\\
\hphantom{\frac{[C_{I_1I_3}, C_{I_3I_1} ]}{q^2-q^{-2}}=}{}+ C_{I_1I_2}(C_{I_1}C_{I_2}+C_{I_3}C_{I_1I_2I_3} ) ,\label{eq:com13} \\
 \frac{[C_{I_1I_3}, C_{I_2I_4} ]}{q^2-q^{-2}}= C_{I_3} C_{I_4}C_{I_{1}I_2}+C_{I_1} C_{I_2}C_{I_3I_4}-C_{I_2} C_{I_3}C_{I_1I_4}-C_{I_1} C_{I_4}C_{I_2I_3} -C_{I_1I_2}C_{I_3I_4}\nonumber\\
\hphantom{\frac{[C_{I_1I_3}, C_{I_2I_4} ]}{q^2-q^{-2}}=}{}+C_{I_2I_3}C_{I_1I_4}.\label{eq:com1324}
 \end{gather}
\end{lemm}
\begin{proof} To prove~\eqref{eq:com13}, let replace $C_{I_3I_1}$ by its definition and express the ($q$-)commutators to obtain
\begin{align*}
[C_{I_1I_3}, C_{I_3I_1} ]={}&{}-[C_{I_1I_3}, [C_{I_2I_3},C_{I_1I_2}]_q ]\\
={}&{}-\frac{1}{q-q^{-1}}\big(qC_{I_1I_3}C_{I_2I_3}C_{I_1I_2}-q^{-1}C_{I_1I_3}C_{I_1I_2}C_{I_2I_3}-qC_{I_2I_3}C_{I_1I_2}C_{I_1I_3}\\
&{}+q^{-1}C_{I_1I_2}C_{I_2I_3}C_{I_1I_3}\big).
\end{align*}
Then, in the four elements of the above sum, we move $C_{I_1I_3}$ in the middle of the products using~\eqref{relaw31}
or~\eqref{relaw32} to get~\eqref{eq:com13}.

Replace $C_{I_2I_4}$ in the right-hand side of~\eqref{eq:com1324} by using~\eqref{eq:C13C31} to get
\begin{align*}
 {\rm r.h.s.}&=\frac{1}{q^2-q^{-2}} \big[C_{I_1I_3}, -q^2C_{I_4I_2}+\big(q^2+1\big)(-C_{I_3I_4}C_{I_2I_3}+C_{I_2}C_{I_4}+C_{I_3}C_{I_2I_3I_4}) \big]\\
 &=\frac{q}{q-1/q} ( -C_{I_1I_3}C_{I_3I_4}C_{I_2I_3}+C_{I_3I_4}C_{I_2I_3}C_{I_1I_3}+C_{I_3}[C_{I_1I_3},C_{I_2I_3I_4} ] ).
\end{align*}
We have used in the second line the result of Lemma~\ref{lemm:relcom2}. By using the commutation relations to move $C_{I_1I_3}$ in the middle of the two first elements of the sum and
\eqref{eq:comut4} to replace the commutator, one gets the results~\eqref{eq:com1324}.
\end{proof}

\begin{rema}%\label{def:awn}
Lemma~\ref{prop:def-nico} allows us to give an alternative definition of $\aw$, where the defining relations~\eqref{relaw31v} and~\eqref{relaw41v}
are replaced by
 \begin{subequations}
\begin{gather}
%\tag{\textbf{R1}}
C_{I_1I_2I_4}=- [C_{I_2I_3},C_{I_1I_3I_4}]_q + C_{I_2}C_{I_1I_4}+C_{I_3}C_{I_1I_2I_3I_4} , \qquad\text{if $I_1=\varnothing$ or $H=\varnothing$} ,\label{Rel-1}\\
%\tag{\textbf{R2}}
C_{I_1I_3I_4}=-[C_{I_1I_2I_4},C_{I_2I_3}]_q +C_{I_1I_4}C_{I_3}+C_{I_2}C_{I_1I_2I_3I_4} ,\qquad\text{if $H=\varnothing$ or $I_4=\varnothing$} ,\label{Rel-2}\\
%\tag{\textbf{R3}}
C_{I_2I_3}=-[C_{I_1I_3I_4},C_{I_1I_2I_4}]_q + C_{I_3}C_{I_2}+C_{I_1I_4}C_{I_1I_2I_3I_4} ,\qquad\text{if $I_4=\varnothing$ or $I_1=\varnothing$}.\label{Rel-3}
\end{gather}
 \end{subequations}
for any monotonic sequence of five adjacent connected subsets $(I_1,I_2,H,I_3,I_4)$, represented by the picture
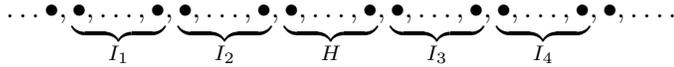
\[\ldots \bullet,\underbrace{\bullet,\dots,\bullet}_{I_1},\underbrace{\bullet,\dots,\bullet}_{I_2},\underbrace{\bullet,\dots,\bullet}_{H},
\underbrace{\bullet,\dots,\bullet}_{I_3},\underbrace{\bullet,\dots,\bullet}_{I_4},\bullet,\ldots .\]
Indeed, when $H=\varnothing$, the four sets $I_1$, $I_2$, $I_3$, $I_4$ become adjacent:~\eqref{Rel-1} and~\eqref{Rel-2} are equivalent to
\eqref{relaw42} and~\eqref{relaw44}. For the other cases, one needs a relabeling of the sets: for $I_1=\varnothing$,
$(I_2,H,I_3,I_4)\to(I_1,I_2,I_3,I_4)$ and for $I_4=\varnothing$,
$(I_1,I_2,H,I_3)\to(I_1,I_2,I_3,I_4)$. Then,~\eqref{Rel-1} for $I_1=\varnothing$,~\eqref{Rel-2} for $I_4=\varnothing$,~\eqref{Rel-3} for $I_1=\varnothing$,~\eqref{Rel-3} for $I_4=\varnothing$ are equivalent to~\eqref{relaw41},~\eqref{relaw43},~\eqref{relaw45},~\eqref{relaw46}, respectively.

This set of relations already appeared in~\cite{DeC} as relations satisfied by the intermediate Casimir elements of ${\rm U}_q(\mathfrak{sl}_2)^{\otimes n}$ and in~\cite{CL} to describe the relations
of the Skein algebra $\text{Sk}_q(\Sigma_{0,n+1})$.
\end{rema}

\subsection[Definition of elements C\_\{I\_1\dots I\_k\}]{Definition of elements $\boldsymbol{C_{I_1\dots I_k}}$}%\label{sec:holes}

In this section, we show how to define unambiguously elements $C_{I_1\dots I_k}$ in $\aw$, for any monotonic sequence $(I_1,\dots, I_k)$ of non-empty connected subsets of $\{1,\dots,n\}$.
We suppose that the holes between $I_{j}$ and $I_{j+1}$, denoted by $H_j$, are non empty.
Visually, we have
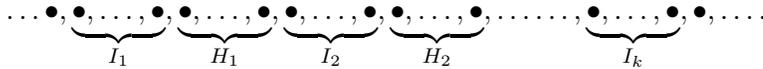
\[\ldots \bullet,\underbrace{\bullet,\dots,\bullet}_{I_1},\underbrace{\bullet,\dots,\bullet}_{H_1},\underbrace{\bullet,\dots,\bullet}_{I_2},\underbrace{\bullet,\dots,\bullet}_{H_2},\ldots\ldots ,\underbrace{\bullet,\dots,\bullet}_{I_{k}},\bullet,\ldots .\]
We define $C_{I_1I_2\dots I_k}$ recursively on $k$.

The case $k=2$ (a single hole) was already dealt with in (\ref{def-CI2}). For $k\geq 2$ arbitrary, we choose a hole, that is we choose $a\in\{1,\dots,k-1\}$ and we set
\begin{equation}\label{def-CIgen}
%\tag{\textbf{def'}}
C_{I_1\dots I_k}:=-[C_{I_{\leq a}H_a},C_{H_aI_{>a}}]_q+C_{I_{\leq a}}C_{I_{>a}}+C_{H_a}C_{I_{\leq a}H_aI_{>a}} ,
\end{equation}
where $I_{\leq a}$ and $I_{>a}$ respectively mean $I_1\dots I_a$ and $I_{a+1}\dots I_k$. All subsets involved in the right-hand side have strictly less than $k-1$ holes,
so that the corresponding elements are already defined in a~preceding step of the recursion.
From this definition, one sees that $C_{I_1\dots I_k}$ is of degree $k$ in the generators. One sees also by an easy recursion that
\[ C_{I_1\dots I_k}=C_{I_k\dots I_1}^{\rm up} .\]
To calculate completely $C_{I_1\dots I_k}$ in terms of the generators, we have to use (\ref{def-CIgen}) $k-1$ times using successively all holes of the sequence.
The following lemma shows that the resulting element does not depend on the order with which the holes are selected.

\begin{lemm}\label{lem:Commg}
 Let $k\geq 2$, $(I_1,\dots, I_k)$ be a monotonic sequence of non-empty connected subsets of $\{1,\dots,n\}$ with non-empty holes between $I_{j}$ and $I_{j+1}$.
\begin{enumerate}\itemsep=0pt
\item[$1.$] The definition~\eqref{def-CIgen} of $C_{I_1\dots I_k}$ is independent of the choice of $a\in\{1,\dots,k-1\}$.
\item[$2.$] Let $I=I_1 \cup \dots \cup I_{k-1}$ and $J$ be a connected set in $\{1,\dots,n\}$.
If $I\cap J=\varnothing$ or $J\subset I$ or~$I\subset J$, then we have $[C_{I_1\dots I_{k-1} },C_J]=0$.
\item[$3.$]
Let $1\leq p \leq k-1$, one gets
\begin{equation}\label{eq:com4}
 [C_{I_1I_2\dots I_p} , C_{I_{p+1}\dots I_{k}}]=0 .
\end{equation}
\end{enumerate}
\end{lemm}
\begin{proof}
We prove the three properties simultaneously by a recursion on $k$.
If $k=2$, item (1) is obvious, while for items (2) and (3), there is no element with hole and they are implied by the defining relations~\eqref{relcommv}.

Let $k>2$ and suppose that the lemma is true for $k'\!<k$. Consider the sequence $(I_1,I_2,\dots, I_{k})$.
Take two holes $H_a$ and $H_{a+1}$ of the sequence. We will show that using $H_a$ or $H_{a+1}$,~\eqref{def-CIgen} provides the same element.
To lighten the presentation, we introduce the following notations:
\begin{equation*}
 I=I_1\dots I_a ,\qquad H=H_a ,\qquad J=I_{a+1} ,\qquad H'=H_{a+1} ,\qquad K=I_{a+2}\dots I_k .
\end{equation*}
Relation~\eqref{def-CIgen} with the hole $H=H_a$ gives
\begin{equation*}
C^{(a)}_{I_1\dots I_k}=-[C_{IH},C_{HJK}]_q+C_{I}C_{JK}+C_{H}C_{IHJK} ,
\end{equation*}
to be compared with
\begin{equation}\label{khole2}
C^{(a+1)}_{I_1\dots I_k}=-[C_{IJH'},C_{H'K}]_q+ C_{IJ}C_{K}+C_{H'}C_{IJH'K} ,
\end{equation}
obtained from relation~\eqref{def-CIgen} with the hole $H'=H_{a+1}$.

Using~\eqref{def-CIgen} for $C_{HJK}$ with the hole $H'$ (and the recurrence hypothesis), one gets
\begin{equation}\label{khole1}
C^{(a)}_{I_1\dots I_k}=\big[C_{IH},[C_{HJH'},C_{H'K}]_q-C_{HJ}C_{K}-C_{H'}C_{HJH'K}\big]_q+C_{I}C_{JK}+C_{H}C_{IHJK} .
\end{equation}
From~\eqref{eq:com4} and the recursion hypothesis, one gets the following commutation relations
\begin{equation*}
 [C_{IH},C_{K}]=[C_{IH},C_{H'}]=[C_{I},C_{H'K}]=[C_{H},C_{H'K}]=[C_{IH},C_{H'K}]=0 .
\end{equation*}
Using the $q$-Jacobi identity and the above commutation relations,~\eqref{khole1} becomes
\begin{gather*}%\label{khole4}
C^{(a)}_{I_1\dots I_k}=\big[[C_{IH},C_{HJH'}]_q,C_{H'K}]\big]_q -[C_{IH},C_{HJ}]_q C_{K}-C_{H'}[C_{IH},C_{HJH'K}]_q \\
\hphantom{C^{(a)}_{I_1\dots I_k}=}{}+C_{I}C_{JK}+C_{H}C_{IHJK} .
\end{gather*}
Then, we use the definition (\ref{def-CIgen}) for a fewer number of holes to expand all the $q$-commutators
and we get the right hand side of~\eqref{khole2}. This proves the point (1) for $k$.

We consider the
sequence $I=I_1\dots I_{k-1}$ and $J$ a connected subset. By recurrence hypothesis, $C_I$ is defined uniquely by~\eqref{def-CIgen} independently of $a$.
Choose $a$ such that $1\leq a \leq k-2$.
If~$I\cap J=\varnothing$, since $J$ is connected, we have either $J\cap H_a=\varnothing$ or $J\subset H_a$.
Then, through the use of~\eqref{def-CIgen}, the commutator $[C_I,C_J]$ reduces to combinations of commutators $[C_K,C_J]$ with
sequences $K$ of at most $k-3$ holes and obeying either $K\cap J=\varnothing$ or $J\subset K$. Thus, they vanish by the recursion hypothesis.

We suppose now that $J\subset I$. Since $J$ is connected, then we have either $J\subset I_{\leq a}$, or $J\subset I_{>a}$.
In all cases, using~\eqref{def-CIgen} again leads to commutators that vanish due to the recursion hypothesis.

When $I\subset J$, the use of~\eqref{def-CIgen} again allows to construct $C_I$ in terms of sequences with less holes, and to conclude using the recursion hypothesis which finishes the proof of item (2)

Item (3) is proven by expressing $C_{I_1I_2\dots I_p}$ and $C_{I_{p+1}\dots I_k}$ only in terms of the generators of $\aw$ by recursively using~\eqref{def-CIgen} and then~\eqref{relcommv}.
\end{proof}

\begin{exam}
Both following expressions of the element $C_{135}$ are equivalent:
\begin{equation*}
 C_{135} =-[C_{134},C_{45}]_q+C_{13}C_{5}+C_{4}C_{1345}=-[C_{12},C_{245}]_q+C_1C_{45}+C_2C_{1245} .
\end{equation*}
\end{exam}

\section[Braid group action as automorphisms on aw(n)]{Braid group action as automorphisms on $\boldsymbol{\aw}$}\label{sec-auto}

In this subsection, we define some maps on the generators of $\aw$ and show that they extend to automorphisms of $\aw$.
Then, we show that they satisfy relations including the defining relations of the braid group. We emphasize that this defines an action on $\aw$ of the braid group on $n+1$ strands.

\subsection[Definition of the maps r\_i]{Definition of the maps $\boldsymbol{r_i}$}

We start by defining an algebra anti-automorphism which is the identity on the generators:
\begin{equation}\label{eq:up}
 {\cdot}^{\rm up}\colon \ C_I\mapsto C_I \quad \text{for any connected subset $I$} .
\end{equation}
To check that ${\cdot}^{\rm up}$ is well defined, note first that the commutation relations (\ref{relcommv}) are obviously preserved. Then the defining relation (\ref{relaw31v}) is sent to relation (\ref{relaw32}) corresponding to $(I_3,I_2,I_1)$. Similarly the defining relation (\ref{relaw41v}) to relation (\ref{relaw43}) corresponding to $(I_4,I_3,I_2,I_1)$.

Note that applying the definition on ${\cdot}^{\rm up}$ on the formula (\ref{def-CI2}) defining elements $C_{I_1I_2}$, we find:
\[C_{I_1I_2}^{\rm up}=C_{I_2I_1}\quad \text{for any disjoint connected subsets $I_1$, $I_2$.}\]
To prove the above relation, we have used the anti-morphism property of ${\cdot}^{\rm up}$ and some commutation relations.

\begin{rema}
Note that there is also an algebra automorphism which is the identity on the generators and sends $q$ to $q^{-1}$. It also sends $C_{I_1I_2}$ to $C_{I_2I_1}$. The verification is similar as for the anti-automorphism ${\cdot}^{\rm up}$.
\end{rema}

We define maps on the generators of $\aw$ extended (keeping the same names) multiplicatively on any product of the generators of $\aw$
and linearly on any linear combination. We prove that this indeed results in well-defined automorphisms of $\aw$ in Theorem~\ref{theo-matricesR}.
 \begin{defi}\label{def:rrb} Let $i\in\{1,\dots,n-1 \}$. The map $r_i$ on the generators of $\aw$ is defined by
\begin{gather}\label{def-rig}
r_i\colon \ \begin{cases}
C_{i+1\dots k} \mapsto C_{i,i+2\dots k} & \text{for }k\geq i+1 ,\\
C_{j \dots i} \mapsto C_{i+1,i-1\dots j} \quad & \text{for }j\leq i ,\\
 C_{j\dots k} \mapsto C_{j\dots k} & \text{if $\{i,i+1\}\subset \{j,\dots,k\}$ or $\{i,i+1\}\cap \{j,\dots,k\}=\varnothing$.}
\end{cases}
\end{gather}
The map $r_0$ on the generators of $\aw$ is defined by
\begin{equation*}%\label{def-r0g}
r_0\colon \
\begin{cases}
C_{\{j\dots k\}}\mapsto C_{\{j\dots k\}} ,& 2\leq j\leq k\leq n , \\
C_{\{12\dots k\}}\mapsto C_{\{k+1\dots n\},1} ,& 1\leq k\leq n .
\end{cases}
\end{equation*}
Finally, for $a\in\{0,1,\dots,n-1\}$, the map $\rb_a$ on the generators of $\aw$ is defined by
\begin{equation*}%\label{shiftig}
 \rb_a(C_I)=(r_a(C_I))^{\rm up}\qquad \text{for any connected subset $I$,}
\end{equation*}
where the involution ${.}^{\rm up}$ is defined in~\eqref{eq:up}.
\end{defi}
Let us remark at once that $r_a$ and $\rb_a$ are related by the involution ${.}^{\rm up}$ in general, namely,
\begin{equation}\label{shifti}
 \rb_a(X^{\rm up})=(r_a(X))^{\rm up} \qquad \text{for any $X\in\aw$.}
\end{equation}
Indeed, this is true when $X=C_I$ is a generator, by definition of $\rb_a$, and thus it is also true for any product of the generators using the antimorphism property of ${.}^{\rm up}$.
\begin{exam}
Note that, for $i=\{1,\dots,n-1\}$, the maps $r_i$ and $\rb_i$ restricted to the central elements $C_1,\dots,C_n$ act as the transposition of $C_i$ and $C_{i+1}$ (and they leave the central element~$C_{1\dots n}$ invariant). The maps $r_0$ and $\rb_0$ exchange $C_1$ with $C_{1\dots n}$ and leave the other $C_i$ invariant.

Here are some examples of the action on non-central generators $C_I$:
\begin{gather*}
r_2(C_{34})=C_{24} ,\qquad\! r_2(C_{12})=r_2(C_{21})=C_{31}\qquad\! \text{and}\qquad\! \rb_2(C_{34})=C_{42} ,\qquad\! \rb_2(C_{12})=C_{13}.
\end{gather*}
For $n=3$, the actions of $r_0$ and $\rb_0$ are
\[
r_0(C_{12})=C_{31} ,\qquad r_0(C_{23})=C_{23}\qquad \text{and}\qquad \rb_0(C_{12})=C_{13} ,\qquad \rb_0(C_{23})=C_{23}.
\]
Note that $r_0$ cannot be expressed in terms of $r_i$, $\rb_i$, $i=1,\dots,n-1$, since $r_0(C_1)=C_{1\dots n}$.
\end{exam}

There exist explicit formulas on arbitrary elements $C_I$. Indeed, even if the image by the maps~$r_a$,~$\rb_a$ of an arbitrary element $C_{I_1\dots I_k}$ is not always another element $C_{J_1\dots J_{k'}}$, there is still an explicit formula.
\begin{prop}\label{prop-ri}
Let $I=(I_1,\dots,I_k)$ be a monotonic sequence of connected subsets.

\begin{enumerate}\itemsep=0pt
\item[$1.$] For $i\in\{1,\dots,n-1\}$, the maps $r_i$, $\rb_i$ act on $C_{I}$ as follows:
\begin{itemize}\itemsep=0pt
 \item If $I\cap\{i,i+1\}$ has exactly one element, denoted $\{a\}=I\cap\{i,i+1\}$, then
\begin{align}\label{def-ri}
&r_i(C_I)=-[C_{i,i+1},C_{I}]_q+C_{ \{i,i+1\}\setminus\{a\} }C_{I\setminus \{a\}}+C_{a}C_{I\cup\{i,i+1\}} ,
\\
\label{def-rbi}
&\rb_i(C_I)= -[C_{I},C_{i,i+1}]_q+C_{ \{i,i+1\}\setminus\{a\} }C_{I\setminus \{a\}}+C_{a}C_{I\cup\{i,i+1\}} .
\end{align}
\item
If $I\cap\{i,i+1\}$ is empty or contains exactly two elements, then \begin{equation} r_i(C_I)=\rb_i(C_I)=C_I. \label{def-rrbi}\end{equation}
\end{itemize}
\item[$2.$] The maps $r_0$, $\rb_0$ act on $C_{I}$ as follows:
\begin{alignat}{3}
& r_0(C_I)=-[C_{2\dots n},C_{I}]_q+C_{ \{2,\dots,n\}\setminus I }C_{1}+C_{I\cap\{2,\dots,n\}}C_{1\dots n} , \quad && \text{if $1\in I$,} & \nonumber\\
& \rb_0(C_I)= -[C_{I},C_{2\dots n}]_q+C_{ \{2,\dots,n\}\setminus I }C_{1}+C_{I\cap\{2,\dots,n\}}C_{1\dots n} , && \text{if $1\in I$,} & \nonumber\\
& r_0(C_I)=\rb_0(C_I)=C_I , && \text{if $1\notin I$.} &\label{def-r0}
\end{alignat}
\end{enumerate}
\end{prop}
\begin{proof}
First, note that the formulas for $\rb_a$, $a=0,1,\dots,n-1$, follow from the ones for $r_a$, using~\eqref{shifti}.

Then, we prove the formulas by recursion on $k$. The case $k=1$ corresponds to a sequence without hole and the actions~\eqref{def-ri}--\eqref{def-r0} correspond immediately to the ones of definition~\ref{def:rrb}, using the Definition~\eqref{def-CI2} of elements $C_{I_1I_2}$ with one hole.

Now, we take $k>1$ and we use the definition~\eqref{def-CIgen} to write
\[C_I=-[C_{I_1H},C_{HI_2}]_q+C_{I_1}C_{I_2}+C_H C_{I_1HI_2} ,\]
with all elements appearing in the right hand side having fewer than $k$ holes. So we can apply the induction hypothesis for the action of $r_i$. When $\{i,i+1\}\subset I$ or $\{i,i+1\}\cap I=\varnothing$, the map $r_i$ acts trivially on each term in the right hand side and the result is immediate. When $\{i,i+1\}\cap I=\{a\}$, the map $r_i$ leaves invariant one term in the $q$-commutator. In this case, the calculation to be done is exactly the same as the one detailed in the appendix around formula~\eqref{app-relKIJ}.

The verification for $r_0$ follows exactly the same steps and we leave it to the reader.
\end{proof}

Thanks to this proposition, we can relate the action of $r_a$, $\rb_a$ to a usual commutator.
\begin{coro}%\label{coro-ri-rbi}
Let $I=(I_1,\dots,I_k)$ be a monotonic sequence of connected subsets. We have
\begin{gather*}%\label{ri-rbi}
r_i(C_I)-\rb_i(C_I)=\displaystyle\frac{q+q^{-1}}{q-q^{-1}}[C_I,C_{i,i+1}] , \qquad i\in\{1,\dots,n-1\} ,\\
r_0(C_I)-\rb_0(C_I)=\displaystyle\frac{q+q^{-1}}{q-q^{-1}}[C_I,C_{2\dots n}] .
\end{gather*}
\end{coro}
\begin{proof}
If $I\cap\{i,i+1\}$ is empty or contains exactly two elements, then $C_{i,i+1}$ commutes with $C_I$, see Lemma~\ref{lem:Commg}. In this case, we also have $r_i(C_I)=\rb_i(C_I)$ so the formula is checked. Similarly, if $1\notin I$, then $C_{2\dots n}$ commutes with $C_I$, again by Lemma~\ref{lem:Commg}, and we also have $r_0(C_I)=\rb_0(C_I)$.

Otherwise, the formulas follow immediately from
\[ [A , B]_q -[B , A]_q=\frac{q+q^{-1}}{q-q^{-1}}[A , B] ,\]
since, in $r_a(C_I)$ and $\rb_a(C_I)$, the right-hand side outside the $q$-commutators coincide.
\end{proof}

The actions of $r_i$ and $\rb_i$ on an element~$C_{I_1\dots I_k}$ may be simple and give another element~$C_{J_1\dots J_{k'}}$. The conditions on $i$ and the sets $I_j$ such that
this happens are given in the following proposition.
\begin{prop}\label{prop-act-ri}$\ $
Let $I=(I_1,\dots,I_k)$ be a monotonic sequence of connected subsets.
\begin{enumerate}\itemsep=0pt
\item[$1.$] Let $i\in\{1,\dots,n-1\}$. The action of $r_i$ on $C_{I_1\dots I_k}$ gives the element
\begin{itemize}\itemsep=0pt
\item where $i+1$ is replaced by $i$ in $I_1\dots I_k$, if the sequence $I_1<\dots<I_k$ is increasing and contains {$i+1$} and not $i$.
More precisely, if $i+1$ is the smallest element of a subset~$I_\ell$, then
\begin{equation*}
 r_i(C_I)= C_{I_1\dots I_{\ell-1}, i, I_{\ell}\setminus \{i+1\}, I_{\ell+1},\dots I_k}.
\end{equation*}

\item where $i$ is replaced by $i+1$ in $I_1\dots I_k$, if the sequence $I_1>\dots>I_k$ is decreasing and contains {$i$} and not $i+1$.
More precisely, if $i$ is the biggest element of a subset $I_\ell$, then
\begin{equation*}
 r_i(C_I)= C_{I_1\dots I_{\ell-1}, i+1, I_{\ell}\setminus \{i\}, I_{\ell+1},\dots I_k}.
\end{equation*}
\end{itemize}
\item[$2.$] If the sequence is increasing: $I_1 < I_2 <\dots< I_k$, and $1\in I_1$. We have
\begin{equation*}
 r_0( C_{I_1I_2\dots I_k} ) = C_{H_k H_{k-1}\dots H_1, 1},
\end{equation*}
where $H_1<\dots < H_k$ is the sequence of connected subsets complementary to $(I_1,\dots,I_k)$ in $\{1,\dots,n\}$.
\end{enumerate}
\end{prop}

\begin{proof}
Let $i\in\{1,\dots,n-1\}$. We detail the proof of the relation when the sequence ${I_1\!<\dots<I_{k}}$ is increasing and contains only $i+1$. The proof for the other case is similar and should be carried at the same time.
We make again a recursion on the number of holes. For $k=1$, we get the definition~\eqref{def-rig}.

We consider $I=I_1\dots I_{\ell-1} I_{\ell}\dots I_{k}$ with $k>1$ and $i+1\in I_{\ell}$. We denote $I'_{\ell}=I_{\ell}\setminus\{i+1\}$.

We assume first that $\ell>2$, and take $H$ to be the hole between $I_1$ and $I_2$. Since $\ell>2$, $H$~does not contain the indices $i$, $i+1$. We write $C_I$ as
\begin{equation*}
C_I = -[C_{I_1H} , C_{HI_{2}\dots I_{k}}]_q+
C_{I_1} C_{I_{2}\dots I_{k}}+C_{H} C_{I_1 HI_{2}\dots I_{k}}.
\end{equation*}
Applying $r_i$, we get by linearity
\begin{equation*}
r_i(C_I) = -[C_{I_1 H} , r_i(C_{HI_{2}\dots I_{k}})]_q+
C_{I_1} r_i(C_{I_{2}\dots I_{k}})+C_{H} r_i(C_{I_1 HI_{2}\dots I_{k}}),
\end{equation*}
where we have used~\eqref{def-rrbi} when $r_i$ acts trivially.
Since the remaining generators have one hole less, we can use the recursion hypothesis to get
\begin{equation}\label{eq:toto}
r_i(C_I) = -\big[C_{I_1 H} , C_{HI_{2}\dots I_{\ell-1},i,I'_{\ell}\dots I_{k}}\big]_q+
C_{I_1} C_{I_{2}\dots I_{\ell-1},i,I'_{\ell}\dots I_{k}}+C_{H} C_{I_1 HI_{2}\dots I_{\ell-1},i,I'_{\ell}\dots I_{k}}.
\end{equation}
One recognises in the right-hand side of~\eqref{eq:toto} the expression of $C_{I_1I_{2}\dots I_{\ell-1},i,I'_{\ell}\dots I_{k}}$.

We consider now the case where $\ell=1$ or $\ell=2$ and $k>2$. We start now with a hole~$H$ between~$I_{k-1}$ and $I_{k}$. By hypothesis $H$ does not contain $i$ nor $i+1$ and we can perform a~calculation analogous to the one done above to end the recursion.

Finally, it remains only to consider the case $k=\ell=2$. In this case, we use an induction on the size of the hole $H$ between $I_1$ and $I_2$. We still have $i+1\in I_2$ and $I'_2=I_2\setminus\{i+1\}$. If the size is $1$, then $H=\{i\}$ and applying $r_i$ on the defining formula for $C_{I_1I_2}$, we find
\[r_i(C_{I_1I_2}) = -[C_{i+1,I_1} , C_{HI_2}]_q+
C_{I_1} C_{HI'_2}+C_{i+1} C_{I_1 HI_{2}},\]
which is indeed equal to $C_{I_1,i,I'_2}$ thanks to relation~\eqref{relaw47} with $(I_1,I_2,I_3,I_4)\to (I_1,\{i\},{\{i+1\}},\allowbreak I'_2)$. Then if the size of $H$ is greater than 1, then we denote $H'=H\setminus\{i\}$ and we write $C_{I_1I_2}$ as follows:
\[C_{I_1I_2} = -[C_{I_1H'} , C_{H'I_2}]_q+
C_{I_1} C_{I_2}+C_{H'} C_{I_1 H'I_{2}},\]
which is relation~\eqref{relaw43}. We can then use the induction hypothesis to apply $r_i$ and find
\[r_i(C_{I_1I_2}) = -\big[C_{I_1H'} , C_{HI'_2}\big]_q+
C_{I_1} C_{i,I'_2}+C_{H'} C_{I_1 HI'_{2}},\]
which is, by~\eqref{def-CIgen}, equal to $C_{I_1,i,I'_2}$ as it should.

We postpone the proof of the relation for $r_0$ to the appendix.
\end{proof}

There are opposite rules for $\rb_i$, $i=1,\dots,n-1$, which are obtained using the relation with the antiautomorphism ${}^{\rm up}$ and read explicitly as:
\begin{itemize}\itemsep=0pt
\item if there exists a subset $I_\ell$ whose the smallest element is $i+1$ in the decreasing sequence $I_1>\dots>I_k$, then
$\rb_i(C_I)= C_{I_1\dots I_{\ell-1}, I_{\ell}\setminus \{i+1\},i, I_{\ell+1},\dots I_k}$;

\item if there exists a subset $I_\ell$ whose the biggest element is $i$ in the increasing sequence $I_1<\dots<I_k$, then
$\rb_i(C_I)= C_{I_1\dots I_{\ell-1}, I_{\ell}\setminus \{i\},i+1, I_{\ell+1},\dots I_k}$.
\end{itemize}
The rule for $\rb_0$ is that if the sequence is decreasing: $I_1 > I_2 >\dots> I_k$, and $1\in I_k$. We have
\begin{equation*}
 \rb_0( C_{I_1I_2\dots I_k} ) = C_{1, H_{k-1} \dots H_{1}H_0},
\end{equation*}
where $H_{k-1}<\dots < H_0$ is the sequence of connected subsets complementary to $(I_1,\dots,I_k)$ in~$\{1,\dots,n\}$.

\begin{exam}
Let $i\in\{1,\dots,n-1\}$. The rules above can be remembered as follow: $r_i$ can only transform the letter $i+1$ into $i$ when the sequence is increasing, and $i$ into $i+1$ when the sequence is decreasing (so that $r_i$ always moves a letter, either $i$ or $i+1$, to the left, keeping the sequence monotonic). For example,
\[
r_2(C_{34})=C_{24} ,\qquad r_2(C_{12})=r_2(C_{21})=C_{31} ,\qquad r_2(C_{35})=C_{25} ,\qquad r_2(C_{52})=C_{53} .
\]
Note that, for example, $r_2(C_{25})$ is not another $C_I$ (since $C_{25}\neq C_{52}$, the rule above cannot apply). Nevertheless, $r_2(C_{25})$ can be computed through the formula in Proposition~\ref{prop-ri} above.

We have a similar (but reversed) rule for $\rb_i$, so that for example,
\[
\rb_2(C_{34})=\rb_2(C_{43})=C_{42} ,\qquad \rb_2(C_{12})=C_{13} ,\qquad \rb_2(C_{53})=C_{52} ,\qquad \rb_2(C_{25})=C_{35} .
\]
Similarly, $\rb_2(C_{52})$ is not another $C_I$ but can be computed through Proposition~\ref{prop-ri} above.
\end{exam}

\subsection{Action of the braid group by automorphisms}

We are now ready to state and prove the main result of this section.
\begin{theo}\label{theo-matricesR}
The maps $r_a$, $\rb_a$ for $a=0,1,\dots,n-1$ are automorphisms of $\aw$ satisfying:
\begin{alignat*}{3}
&r_a\rb_a=\rb_ar_a={\rm Id} , && a=0,1,\dots,n-1, &\\
&r_ar_{a+1}r_a=r_{a+1}r_ar_{a+1} ,\qquad && a=0,1,\dots,n-2, &\\
&r_ar_b=r_br_a , && a,b=0,1,\dots,n-1\ \text{with}\ |a-b|>1 , &
\end{alignat*}
which are the defining relations of the braid group on $n+1$ strands.
\end{theo}
\begin{proof}
The difficult part of the proof is that the maps $r_a$, $\rb_a$ are indeed morphisms. The details of this part are given in appendix. Once this is known, it remains only to verify the relations stated in the theorem for the maps $r_a$, $\rb_a$. Since they are morphisms, it is enough to check the relations on the generators of $\aw$.

The relations $r_a\rb_a=\rb_ar_a={\rm Id}$ is an immediate consequence of Proposition~\ref{prop-act-ri} (which was proved also for $r_0$ in the preceding step, see Appendix~\ref{sec-appendix}). Using this proposition, it is also straightforward to check that both sides of the 3-term braid relation gives the same result when applied on a generator. This is immediate on an element $C_I$ with $I$ of size $1$, and otherwise we have, for the non-trivial cases if $i\in\{1,\dots,n-2\}$:
\[r_ir_{i+1}r_i=r_{i+1}r_ir_{i+1}\colon \ \begin{cases}
C_{i+2\dots k} \mapsto C_{i,i+3\dots k} & \text{for }k\geq i+3 , \\
C_{i+1\dots k} \mapsto C_{i,i+1,i+3\dots k} & \text{for }k\geq i+2 , \\
C_{j\dots i+1} \mapsto C_{i+2,i+1,i-1,\dots,j} & \text{for }j\leq i ,\\
C_{j\dots i} \mapsto C_{i+2,i-1\dots j} & \text{for }j\leq i-1 .\\
\end{cases}\]
If $i=1$, the non-trivial cases are
\[r_0r_{1}r_0=r_{1}r_0r_{1}\colon \ \begin{cases}
C_{2\dots j} \mapsto C_{n \dots j+1, 21} & \text{for }j\geq 2 , \\
C_{1\dots j} \mapsto C_{n\dots j+1,2} & \text{for }j\geq 2 .
\end{cases}\]
Then let $a<b$ with $|a-b|>1$. Let $I$ be a connected subset. If $I$ is such that one of the maps~$r_a$ or $r_b$ (say $r_a$) leaves invariant $C_I$, then it is easy to see that $r_b(C_I)$ is still invariant by $r_a$, so that the relation $r_ar_b=r_br_a$ is verified. The only remaining case is when $I=\{a+1,\dots,b\}$. In this case, it is easy to check, using again Proposition~\ref{prop-act-ri}, that we have
\[r_a\rb_b=\rb_br_a\colon \ C_{a+1\dots b}\mapsto \begin{cases} C_{a,a+2\dots b-1,b+1} & \text{if $a>0$,}\\
C_{n\dots b+2, b, 1} & \text{if $a=0$.}
\end{cases}\]
Therefore, we have $r_a\rb_b=\rb_br_a$ which is equivalent to $r_ar_b=r_br_a$.
\end{proof}

\subsection{Quotient by the centre of the braid group}

At this point, we have an action of the braid group on $n+1$ strands by automorphisms on~$\aw$. This action is not faithful, which means that there are more relations satisfied by the maps~$r_a$,~$\rb_a$, not implied by the relations in Theorem~\ref{theo-matricesR}. To give additional relations, consider the following elements:
\begin{gather}\label{def-Delta}
\Delta_{a,\dots,b}=r_a\cdot r_{a+1}r_a\cdot \ldots \cdot r_{b}\ldots r_{a+1}r_a=r_ar_{a+1}\ldots r_{b}\cdot\ldots\cdot r_a r_{a+1}\cdot r_a \quad \text{for $a<b$,}
\end{gather}
where the equality between the two expressions is easily obtained using the braid relations. It is well known~\cite{KT} that $(\Delta_{a,\dots,b})^2$ generates
the centre of the braid group generated by $r_{a},\ldots, r_{b}$ as long as $b-a\geq 1$.

\begin{prop}\label{prop-Delta}
As automorphisms of $\aw$, the maps $r_0,r_1,\dots,r_{n-1}$ satisfy:
\[(\Delta_{0,\dots,n-1})^2=(\Delta_{0,\dots,n-2})^2=(\Delta_{1,\dots,n-1})^2={\rm Id} .\]
These relations can equivalently be written as:
\begin{gather}
(\Delta_{1,\dots,n-1})^2={\rm Id} , \qquad
r_{n-1}\dots r_1 r_0^2r_1\dots r_{n-1}={\rm Id}.\label{rel-quotient-braid}
\end{gather}
This allows to express $r_0^2$ in terms of the other automorphisms as:
\[
r_0^2=\rb_1\dots \rb_{n-2}\rb_{n-1}^2\rb_{n-2}\dots \rb_1=(\Delta_{2,\dots,n-1})^2 .
\]
\end{prop}
\begin{proof}
Note that all these relations are satisfied on the central elements $C_1,\dots,C_n,C_{1,\dots,n}$ since restricted on this stable subset, the automorphisms $r_a$ are all involutions.

We first describe explicitly the action of $\Delta_{1,\dots,n-1}$ on a generator $C_I$. Let $I=\{i,\dots,j\}$ with~$i< j$. We have
\begin{equation*}%\label{action-Delta1}
\Delta_{1,\dots,n-1}(C_{i\dots j})=C_{n-j+1,\dots,n-i+1} .
\end{equation*}
To prove this formula, we note first that
\[r_1\dots r_k (C_{i\dots j})=\begin{cases}
C_{i+1\dots j+1} & \text{if $j\leq k$,}\\
C_{1,i+1\dots j} & \text{if $i-1\leq k\leq j-1$,}\\
C_{i\dots j} & \text{if $k<i-1$.}
\end{cases}
\]
This is easily checked using Proposition~\ref{prop-act-ri}. Now, when applying $\Delta_{1,\dots,n-1}$ (we use the second formula for $\Delta_{1,\dots,n-1}$ in~\eqref{def-Delta}) on $C_{i\dots j}$, the element $C_{i\dots j}$ remains invariant until reaching the string $r_1\dots r_{i-1}$. Then (after acting with the strings $r_1\dots r_{i-1}$ to $r_1\dots r_{j-1}$) it becomes $C_{1\dots j-i}$, just before reaching the string $r_1\dots r_j$. The $n-j$ last strings (from $r_1\dots r_j$ to $r_1\dots r_{n-1}$) send it to $C_{n-j+1,\dots,n-i+1}$, which is the desired result.

At this point, this makes it clear that $(\Delta_{1,\dots,n-1})^2$ is the identity. We move on to calculating $\Delta_{0,\dots,n-1}$. We have, still using Proposition~\ref{prop-act-ri}:
\[\Delta_{0,\dots,n-1} (C_{i\dots j})=r_{n-1}\dots r_1r_0\Delta_{1,\dots,n-1} (C_{i\dots j})=\begin{cases}
C_{n-j\dots n-i} & \text{if $j\neq n$,}\\
C_{n-i+1\dots n} & \text{if $j=n$.}
\end{cases}
\]
Here also, this makes it clear that $(\Delta_{0,\dots,n-1})^2$ is the identity. At this point, relations~\eqref{rel-quotient-braid} are implied using:
\[\Delta_{0,\dots,n-1}=r_{n-1}\dots r_{1}r_{0} \Delta_{1,\dots,n-1}= \Delta_{1,\dots,n-1}r_0r_1\dots r_{n-1} .\]
Finally, we get that $\Delta_{0,\dots,n-2}$ squares to the identity from the already obtained relations and
\[\Delta_{0,\dots,n-2}=\rb_{n-1}\dots \rb_{1}\rb_{0} \Delta_{0,\dots,n-1}= \Delta_{0,\dots,n-1}\rb_0\rb_1\dots \rb_{n-1} .\]
Using that $\Delta_{1,\dots,n-1}=r_{n-1}\dots r_{1} \Delta_{2,\dots,n-1}= \Delta_{2,\dots,n-1}r_1\dots r_{n-1}$ and $(\Delta_{1,\dots,n-1})^2={\rm Id}$, the second equality in the expression from $r_0$ follows.
\end{proof}

\begin{rema}
Let $n=3$ and consider only the automorphisms $r_1$, $r_2$, $\rb_1$, $\rb_2$. The preceding proposition implies that these automorphisms give an action of the braid group on 3 strands,
which factors through the quotient by its centre, generated by $(r_1r_2r_1)^2$.
For $n=3$, this quotient of the braid group is also called the modular group and is isomorphic to ${\rm PSL}_2(\mathbb{Z})$. This automorphism group has been already found in~\cite{Ter} for the universal Askey--Wilson algebra,
and moreover in this case the action is faithful meaning that there is no more independent relations satisfied by the maps $r_1$, $r_2$.

For arbitrary $n\geq 3$, considering only the automorphisms $r_i$, $\rb_i$ with $i=1,\dots,n-1$, we have an action of the braid group on $n$ strands, which factors through its quotient by its centre,
thus generalising the result of~\cite{Ter} for $n=3$. We do not know if the action is faithful for this group.
\end{rema}

\begin{rema}
Even for $n=3$, we have a new automorphism $r_0$, which produces an action of the braid group on 4 strands on $\mathfrak{aw}(3)$.
This action also factors through the quotient by the centre, here generated by $(r_2r_1r_0r_2r_1r_2)^2$.
We have more relations since $r_0^2$ is expressed in terms of the other generators. For $n=3$, the formula for $r_0^2$ is simply $r_0^2=r_2^2$.

For any $n$, adding $r_0$, we have an action of the braid group on $n+1$ strands, which also factors through its quotient by the centre. We also have more relations since $r_0^2$ is expressed in terms of the other generators. Note that $r_0$ is never in the subgroup generated by $r_1,\dots,r_{n-1}$, as can be seen from the formula $r_0(C_1)=C_{1\dots n}$.
\end{rema}

\begin{rema}
It may be clear that we must also have an automorphism $r'_0$ which is similar to~$r_0$, but acts on the index $n$ instead of $1$. Indeed, we can define the following map:
\[r'_0\colon \
\begin{cases}
C_{\{j\dots k\}}\mapsto C_{\{j\dots k\}} ,& 1\leq j\leq k\leq n-1, \\
C_{\{j\dots n\}}\mapsto C_{\{1\dots j-1\},n} ,& 1\leq j\leq n.
\end{cases}\]
This extends to an automorphism and it satisfies $r'_0r_ir'_0=r_ir'_0r_i$ if $i=0$ or $i=n-1$, and commutes with the other $r_i$'s. So altogether, the maps $r_0,r_1,\dots,r_{n-1},r'_0$ generate the Artin braid group associated to the affine Dynkin diagram of type A. However, it is not so useful to consider this additional automorphism $r'_0$ since it can expressed in terms of the others. We have
\[r'_0=\rb_{n-1}\dots \rb_1\rb_0\rb_1\dots \rb_{n-1} .\]
This can easily be checked directly on the generators. Note that the braid relations involving $r'_0$ thus follows from the relations in Proposition~\ref{prop-Delta}.
\end{rema}

\subsection{Consequences of the automorphisms}

From the fact that $r_i$ and $\rb_i$ are automorphisms of $\aw$, we can deduce other relations between the elements of the Askey--Wilson algebra.
As stated in the following proposition, relations established in the previous lemmas can be generalised to the cases where the sets are not necessarily adjacent.
\begin{prop}\label{prop:gen}
 For any monotonic sequence of non-empty subsets $(I_1,I_2,I_3,I_4)$, the relations of Lemmas~$\ref{prop:def-nico}$,~$\ref{lem:relc1}$,~$\ref{lemm:relcom2}$ and~$\ref{lem:ff}$
 which contain only increasing $($resp. decreasing$)$ sequence also hold in $\aw$
\end{prop}
\begin{proof}
Let us prove the proposition for an increasing sequence of subsets $I_\ell=\{i_\ell,i_\ell+1,\dots ,j_\ell\}$ with $i_\ell\leq j_\ell< i_{\ell+1}$.
We introduce the subsets $\overline{I_3}=\{i_3+i_4-j_3-1, \dots ,i_4-1\}$, $\overline{I_2}=\{i_2+i_3+i_4-j_2-j_3-2,\dots, i_3+i_4-j_3-2\}$ and $\overline{I_1}=\{i_1+i_2+i_3+i_4-j_1-j_2-j_3-3,\dots ,i_2+i_3+i_4-j_2-j_3-3\}$ such that
$(\overline{I_1},\overline{I_2},\overline{I_3},I_4)$ is a sequence of adjacent subsets and $\# \overline{I_\ell}=\# I_\ell$.
Then, the relations of the lemmas hold for $(\overline{I_1},\overline{I_2},\overline{I_3},I_4)$.
We consider the ones which contain only increasing sequences (for example, second relation of~\eqref{rel:com1} is excluded).
Acting with the automorphisms $r_{i_1+i_2+i_3+i_4-j_1-j_2-j_3-4}$, then $r_{i_1+i_2+i_3+i_4-j_1-j_2-j_3-5}$ up to $r_{i_1}$ we bring the smallest index of $\overline{I_1}$ to $i_1$ (which is the smallest index of $I_1$). Iterating the process, we change the sequence $(\overline{I_1},\overline{I_2},\overline{I_3},I_4)$ to $(I_1,I_2,I_3,I_4)$, so that these relations hold also
for this sequence. The remaining relations are established
 using the just proved ones, and following the proofs of the lemmas.
\end{proof}

From now one, we use the relations of the lemmas in the general setting of Proposition~\ref{prop:gen} without mentioning it.

\section[Coproduct maps on aw(n)]{Coproduct maps on $\boldsymbol{\aw}$}\label{sec-coproduct}

Let $a\in\{0,1,\dots,n\}$ and $1\leq j\leq k\leq n$. We consider the following map:
\begin{equation}\label{eq:delta}
\delta_a\colon \
C_{j\dots k} \mapsto
\begin{cases}
C_{j\dots k}&\mbox{when } k<a ,\\
 C_{j\dots k\,k+1}&\mbox{when } j\leq a\leq k ,\\
 C_{j+1\dots k+1}&\mbox{when } a<j .
 \end{cases}
\end{equation}
Note that $\delta_0$ corresponds simply to increase the indices by one. We call the maps $\delta_a$ coproduct maps. The terminology will be clear when considering the tensor products ${\rm U}_q(\mathfrak{sl}_2)^{\otimes n}$.
\begin{prop}
 Let $a\in\{0,1,\dots,n\}$. The map $\delta_a$ defines a morphism of algebras from $\aw$ to $\mathfrak{aw}(n+1)$.
\end{prop}
\begin{proof}
Due to the form of the defining relations of $\aw$ involving connected subsets of $\{1,\dots,n\}$, it is immediate to check that the maps $\delta_a$ preserve them.
\end{proof}

\begin{exam}
One gets from the definition: $\delta_2(C_1)=C_1$, $\delta_2(C_{12})=C_{123}$, $\delta_2(C_{3})=C_{4}$.
\end{exam}

\subsection{Relations between the coproduct and the braid group action}

Below we give the relations between the coproduct maps and the automorphisms forming the action of the braid group. They actually reflect the quasi-triangularity of ${\rm U}_q(\mathfrak{sl}_2)$ when the algebra $\aw$ is realised
in~${\rm U}_q(\mathfrak{sl}_2)^{\otimes n}$, see a later section. However, the realisation of $\aw$ in ${\rm U}_q(\mathfrak{sl}_2)^{\otimes n}$ is not faithful, and thus it is remarkable that the relations below are already satisfied in $\aw$
before taking the quotient corresponding to its realisation in ${\rm U}_q(\mathfrak{sl}_2)^{\otimes n}$.
\begin{prop}\label{prop-r-delta}
We have the following identities for morphisms from $\aw$ to $\mathfrak{aw}(n+1)$:
\begin{alignat*}{4}%\label{eq:delta-r}
&r_i\delta_i=\delta_i; &&\quad \delta_i=\rb_i\delta_i && \quad\text{for $i=0,1,\dots,n$,} &\\
&\delta_ir_i=r_{i+1}r_i\delta_{i+1}; &&\quad \rb_i\rb_{i+1}\delta_i=\delta_{i+1}\rb_i \quad && \quad\text{for $i=0,1,\dots,n-1$,} &\\
&\delta_{i+1}r_i=r_{i}r_{i+1}\delta_{i}; &&\quad \rb_{i+1}\rb_i\delta_{i+1}=\delta_{i}\rb_i && \quad\text{for $i=0,1,\dots,n-1$,} &\\
&\delta_ir_j=r_j\delta_i; &&\quad \delta_i\rb_j=\rb_j\delta_i && \quad\text{for $i=0,1,\dots,n$ and $j<i-1$,} &\\
&\delta_ir_j=r_{j+1}\delta_i; &&\quad \delta_i\rb_j=\rb_{j+1}\delta_i && \quad\text{for $i=0,1,\dots,n$ and $j>i$.} &
\end{alignat*}
\end{prop}
\begin{proof}
All the equalities need only to be checked on the generators $C_I$ of $\aw$, with $I$ a~connected subset of $\{1,\dots,n\}$. All these verifications are straightforward.
\end{proof}

\subsection[On the definition of aw(n)]{On the definition of $\boldsymbol{\aw}$}\label{sect:aw3-n}
We give a more conceptual equivalent definition of the algebra $\aw$, which puts forward the role of the coproduct maps $\delta_i$ and of the automorphisms $r_i$. In the following discussion, we can and we will ignore the index $0$ for the coproduct maps and the automorphisms.

For any $n\geq 1$, we consider the algebra $\aw$ as generated by elements $C_I$ with $I$ any connected subset of $\{1,\dots,n\}$. Let us introduce the natural map:
\begin{align*} \iota_n\colon \ \aw & \to \mathfrak{aw}(n+1),\\
C_I & \mapsto C_I \quad \text{($I$ a connected subset of $\{1,\dots,n\}$),}
\end{align*}
and recall that we have the coproduct maps $\delta_1,\dots,\delta_n$ defined in~\eqref{eq:delta}. We make the following requirements:
\begin{itemize}\itemsep=0pt
\item[1.] For $n\geq 1$, the map $\iota_n$ and the coproduct maps are morphisms from $\mathfrak{aw}(n)$ to $\mathfrak{aw}(n+1)$.
\item[2.] For $n\geq 1$, there exist automorphisms $r_1,\dots,r_{n-1}$ of $\aw$, commuting with $\iota_n$, and satisfying the conditions in Proposition~\ref{prop-r-delta} (where $\rb_1,\dots,\rb_{n-1}$ denote the inverses).
\item[3.] For $n=2$, $\mathfrak{aw}(2)$ is commutative, and the automorphism $r_1$ exchanges $C_1$ and $C_2$.
\item[4.] Finally, for $n=3$, we impose the following formulas for the automorphism $r_1$:
\begin{gather}\label{axiomAW3}
r_1(C_{23})=-[C_{12},C_{23}]_q+C_1C_3+C_2C_{123}, \nonumber\\
\rb_1(C_{23})=-[C_{23},C_{12}]_q+C_1C_3+C_2C_{123} ,
\end{gather}
that is, with our notations, $r_1(C_{23})=C_{13}$ and $\rb_1(C_{23})=C_{31}$.
\end{itemize}
Note that the condition of commuting with $\iota_n$ in the second item simply means that if $r_i(C_I)$ is calculated in $\aw$, the same result will be true in $\mathfrak{aw}(n')$ for any $n'>n$.

Now let us discuss the meaning of the above requirements step by step:
\begin{itemize}\itemsep=0pt
\item When considering Hopf algebras, the first two items encode the quasitriangularity. Namely, let $A$ be a quasitriangular Hopf algebra and $C$ be any element of $A$. We can define elements $C_I$ in $A^{\otimes n}$ obtained by repeated applications of the coproduct map of $A$ on $C$. Then the first two items will be automatically satisfied if we take the subalgebra in $A^{\otimes n}$ generated by these elements $C_I$. The maps $\delta_i$ are realised by
\begin{equation*}
\Delta_i\colon \ A^{\otimes n} \to A^{\otimes (n+1)},
\end{equation*}
which is the coproduct $\Delta$ applied in the $i^{\rm th}$ copy of $A$. The automorphisms $r_i$, $i=1,\dots,n-1$, are realised by
\begin{align*}
\rho_i\colon \ A^{\otimes n} &\to A^{\otimes n},\\
 X & \mapsto \tau_{i,i+1}\big( R_{i,i+1} X R_{i,i+1}^{-1}\big),
\end{align*}
where $R$ is the $R$-matrix and $\tau_{i,i+1}$ is the flip operation between the $i^{\rm th}$ and $(i+1)^{\rm th}$ copies of $A$. Then the relations in Proposition~\ref{prop-r-delta} becomes direct translations of the quasitriangularity of $A$.

\item Now the third item will be automatically satisfied if we take for $C$ a central element of $A$ (such as for example the Casimir element of a quantum group ${\rm U}_q(\mathfrak{g})$). This is immediate to check.

\item The final requirement is the only one which is very specific. It comes from a relation which was shown to be satisfied when $C$ is the Casimir element of ${\rm U}_q(\mathfrak{sl}_2)$~\cite{CGVZ}.
\end{itemize}

The first key point about these requirements is that they completely fix the action of the automorphisms $r_1,\dots,r_{n-1}$ and their inverses on the generators. This is straightforward to check, recursively on $n$, and we will only show how this works for $n=3$, leaving the remaining details for the reader. First, we apply the relations
\[
r_1\delta_1=\delta_1\qquad \text{and}\qquad r_2\delta_2=\delta_2 ,
\]
on all generators $C_1$, $C_2$, $C_{12}$. We get that $r_1$ leaves invariant $C_{3}$, $C_{12}$, $C_{123}$ and that $r_2$ leaves invariant $C_1$, $C_{23}$, $C_{123}$. Since we already know that $r_1$ exchanges $C_1$ and $C_2$ and the action of~$r_1$,~$\rb_1$ on $C_{23}$ (the last axiom), we are done for $r_1$ and its inverse. Then, we use the relations
\[
\delta_1r_1=r_2r_1\delta_2\qquad \text{and}\qquad \delta_2 r_1=r_1r_2\delta_1 .
\]
This gives the remaining actions of $r_2$ and its inverse, namely that $r_2$ exchanges $C_2$ and $C_3$ and~$r_2(C_{12})=C_{31}$ and $\rb_2(C_{12})=C_{13}$.

Then, the second implication of these requirements is that all relations appearing in the definition of $\aw$ are necessary. For $n=3$, the commuting relations are easy to obtain from the commutativity of $\mathfrak{aw}(2)$ by applying the coproduct maps and some suitable automorphisms. Then, using the known actions of the automorphisms, we apply $r_1r_2$ on the first relation in~\eqref{axiomAW3} and $\rb_2$ on the second relation in~\eqref{axiomAW3}, and we recover the usual defining relations of $\mathfrak{aw}(3)$ as discussed in Example~\ref{ex:aw3}. For $n>3$, it is straightforward to get all defining relations from those for $n=3$ by repeated applications of coproduct maps and automorphisms.

Since we know from the results in the preceding sections that our definition is also sufficient to satisfy the above requirements, we conclude with the following conceptual view on the algebras~$\aw$.
\begin{prop}%\label{prop:axioms}
The defining relations in Definition~$\ref{prop:def-nico2}$ of the algebras $\aw$ are necessary and sufficient to satisfy the above requirements $(1)$ to $(4)$. In other words, the algebras $\aw$ are the largest algebras satisfying these requirements, and for any other sequence of algebras $A_n$ satisfying these requirements, we have that $A_n$ is a quotient of $\aw$.
\end{prop}

\section{Casimir elements}\label{sec-Cas}
In this section, we use the known central element of $\mathfrak{aw}(3)$ and the machinery of coproducts and braid group automorphisms to produce a family of central elements of $\aw$ for any $n$.

\medskip
Let $(I_1,I_2,I_3)$ be a sequence of subsets such that $I_1<I_2<I_3$ (not necessarily adjacent) where each $I_j$ is seen as an increasing sequence of connected subsets. We define the following element of $\aw$
\begin{gather}
 \varOmega_{I_1,I_2,I_3} = qC_{I_1I_2}C_{I_2I_3}C_{I_1I_3}+\frac{q^{2}C_{I_1I_2}^{2}+q^{-2}C_{I_2I_3}^{2}+q^{2}C_{I_1I_3}^{\,2}+C_{I_1I_2I_3}^{2} +
 C_{I_1}^{2}+C_{I_2}^{2}+C_{I_3}^{2} }{q+q^{-1}}\nonumber \\
\hphantom{\varOmega_{I_1,I_2,I_3} =}{}-qC_{I_1I_2}(C_{I_1}C_{I_2}+C_{I_3}C_{I_1I_2I_3}) -q^{-1}C_{I_2I_3}(C_{I_2}C_{I_3}+C_{I_1}C_{I_1I_2I_3}) \nonumber \\
\hphantom{\varOmega_{I_1,I_2,I_3} =}{}
-qC_{I_1I_3}(C_{I_1}C_{I_3}+C_{I_2}C_{I_1I_2I_3})+(q+q^{-1})C_{I_1}C_{I_2}C_{I_3}C_{I_1I_2I_3}
-\frac{1}{q+q^{-1}} ,\label{eq:aw3cas}
\end{gather}
called Casimir elements. This name comes from the fact that these elements are central in $\aw$ as shown below.
The last term in~\eqref{eq:aw3cas} ensures that $\varOmega_{I_1,I_2,I_3}=0$ whenever at least one of the subsets $I_1$, $I_2$ or $I_3$ is empty.

Note that the coproduct maps are easy to apply on the elements $\varOmega_{I_1,I_2,I_3}$. We have immediately:
\begin{equation}\label{eq:deltaOmega}
\delta_i(\varOmega_{I_1,I_2,I_3})=\varOmega_{I'_1,I'_2,I'_3} ,
\end{equation}
where $I'_a$ ($a=1,2,3$) is obtained from $I_a$ by increasing by 1 all elements strictly greater than $i$, and furthermore adding $i+1$ next to $i$ if $i\in I_a$.

\subsection[The central element of aw(3)]{The central element of $\boldsymbol{\mathfrak{aw}(3)}$}

In $\mathfrak{aw}(3)$, there is only one Casimir element denoted $\varOmega_{1,2,3}$.
It was first introduced in~\cite{Zh2}, see also~\cite{avatar, Zh}. We record its main properties in the following proposition.

\begin{prop}\label{propcas3}
The element $\varOmega_{1,2,3}$ is central in $\mathfrak{aw}(3)$ and is invariant by $r_1$, $\rb_1$, $r_2$, $\rb_2$.
\end{prop}
\begin{proof}
The centrality of $\varOmega_{1,2,3}$ is proven in~\cite{Zh2}. Since $\rb_i$ is the inverse of $r_i$, it is enough to show the property for $\rb_1$ and $r_2$.
Their actions on $\varOmega_{1,2,3}$ are easy to compute. Then, algebraic manipulations using the $\mathfrak{aw}(3)$ relations show that we get back to $\varOmega_{1,2,3}$.
\end{proof}

\subsection[The central elements of aw(4)]{The central elements of $\boldsymbol{\mathfrak{aw}(4)}$}

In $\mathfrak{aw}(4)$, the possible choices of $I_1$, $I_2$, $I_3$ lead to the following list of Casimir elements:
 \begin{equation}\label{listcasaw4}
 \varOmega_{1,2,3},\ \varOmega_{1,2,4},\ \varOmega_{1,3,4},\ \varOmega_{2,3,4},\ \varOmega_{12,3,4},\ \varOmega_{1,23,4},\ \varOmega_{1,2,34} .
 \end{equation}
We have the following properties.
\begin{prop}\label{propcas4}
 The elements in~\eqref{listcasaw4} are all central in $\mathfrak{aw}(4)$ and satisfy
\begin{eqnarray}\label{eq:cas4}
 \varOmega_{12,3,4}-\varOmega_{1,3,4}-\varOmega_{2,3,4}=\varOmega_{1,23,4}-\varOmega_{1,2,4}-\varOmega_{1,3,4}=\varOmega_{1,2,34}-\varOmega_{1,2,3}-\varOmega_{1,2,4} .
\end{eqnarray}
\end{prop}
\begin{proof}
We do not have a simple proof for the equalities in~\eqref{eq:cas4} and for the fact that $\varOmega_{1,2,3}$ commutes with $C_{34}$.
For these two facts, we rely on computer-aided calculations~\cite{form}, that we use in the following way.
We have implemented the relations~\eqref{rel:com1}--\eqref{eq:2h2},~\eqref{relaw49},~\eqref{relaw410},~\eqref{relaw42}--\eqref{eq:2h1},~\eqref{eq:comut3}--\eqref{eq:comut4},
\eqref{eq:comut6}--\eqref{eq:comut7},~\eqref{rel:coma3},~\eqref{rel:coma4} and~\eqref{eq:com1324} which allow us to order the elements of the following set
\[\cG=\{C_{134}, C_{14}, C_{124}, C_{24}, C_{13}, C_{234}, C_{123}, C_{34}, C_{23},C_{12}\}.\]
Therefore, any word $W\in \mathfrak{aw}(4)$ can be written as follows:
\begin{gather*}
 W\!= \!\!\!\sum_{i_1,\dots i_{10}}\!\!\! a_{i_1\dots i_{10} }(C_{134})^{i_1} (C_{14})^{i_2} (C_{124})^{i_3} (C_{24})^{i_4} (C_{13})^{i_5} (C_{234})^{i_6} (C_{123})^{i_7} (C_{34})^{i_8} (C_{23})^{i_9} (C_{12})^{i_{10}}.
\end{gather*}
The diamond lemma checks the associativity of the product in the algebra. It consists in choosing three elements $A$, $B$, $C$ of $\cG$ in the wrong order and to order them following two different ways:
firstly we order $AB$, then we order the result with $C$, secondly we order $BC$, and then we order the result with $A$.
This provides relations between ordered monomials. We have computed also further relations by using the diamond lemma with these new relations and the ordering ones.
This set of relations allows to check~\eqref{eq:cas4} and that $\varOmega_{1,2,3}$ commutes with $C_{34}$.

At this stage, we know that $\varOmega_{1,2,3}$ commutes with $C_{12}$, $C_{23}$, $C_{123}$, from Proposition~\ref{propcas3}, and with $C_{34}$. This implies that it commutes with $C_{124}$, using the definition of $C_{124}$ in terms of~$C_{123}$ and $C_{34}$. Then since $C_{234}=\rb_1\rb_2(C_{124})$ and since $\rb_1$, $\rb_2$ leave $\varOmega_{1,2,3}$ invariant, it follows that~$\varOmega_{1,2,3}$ commutes also with $C_{234}$
and is therefore central in $\mathfrak{aw}(4)$.

Now we can obtain all others $\varOmega_{i,j,k}$ from $\varOmega_{1,2,3}$ by applying some automorphisms as follows:
\[\varOmega_{1,2,4}=\rb_3(\varOmega_{1,2,3}) ,\qquad \varOmega_{1,3,4}=\rb_2(\varOmega_{1,2,4}) ,\qquad \varOmega_{2,3,4}=\rb_1(\varOmega_{1,3,4}) .\]
These are directly obtained from the definition of $\varOmega_{i,j,k}$ and the explicit formulas for the action of the automorphisms $r_i$ in Proposition~\ref{prop-act-ri}. So we have that all $\varOmega_{i,j,k}$ are central in $\mathfrak{aw}(4)$.

Finally, we note that, by definition of the coproduct maps, we have
\[\varOmega_{12,3,4}=\delta_1(\varOmega_{1,2,3}) ,\qquad \varOmega_{1,23,4}=\delta_2(\varOmega_{1,2,3}) ,\qquad\varOmega_{1,2,34}=\delta_3(\varOmega_{1,2,3}) ,\]
and that these three elements coincide, due to the equalities~\eqref{eq:cas4} up to elements which are already central. So for any generator $C_I$ of $\mathfrak{aw}(4)$, we only need to show that it commutes with one of these three elements. We can assume that $|I|>1$ (otherwise $C_I$ is central), and thus we have that $C_I=\delta_i(C_J)$ for some $i$ and some $J\subset\{1,2,3\}$. Applying $\delta_i$ to the relation $[\varOmega_{1,2,3},C_J]=0$ of $\mathfrak{aw}(3)$, we get that $C_I$ commutes with one of the elements above, as needed.
\end{proof}

We associate a central element $\omega_S$ to any subset $S\subset\{1,2,3,4\}$ with $|S|\geq3$:
\begin{itemize}\itemsep=0pt
\item for $S=\{a,b,c\}$, we set $\omega_{\{a,b,c\}}=\varOmega_{a,b,c}$ with $a<b<c$;
\item for $S=\{1,2,3,4\}$,
the relations~\eqref{eq:cas4} allow to define a unique element:
\begin{align*}
\omega_{\{1,2,3,4\}}&{}=\varOmega_{12,3,4}-\varOmega_{1,3,4}-\varOmega_{2,3,4}=\varOmega_{1,23,4}-\varOmega_{1,2,4}-\varOmega_{1,3,4}\\
&{}=\varOmega_{1,2,34}-\varOmega_{1,2,3}-\varOmega_{1,2,4} .
\end{align*}
\end{itemize}
The following result gives the action of the braid group automorphisms on the central elements~$\omega_S$ and shows that this action simply amounts to the permutation action of the symmetric group on the subset $S$. Below, $(i,i+1)$ denotes the transposition of $i$ and $i+1$.
\begin{prop}
For all $S\subset \{1,2,3,4\}$ with $|S|\geq 3$, and all $i=1,2,3$, we have
\[r_i\omega_S=\rb_i\omega_S=\omega_{(i,i+1)\cdot S} .\]
\end{prop}
\begin{proof}
Several actions of the automorphisms on the elements $\varOmega_{a,b,c}$ are immediate to obtain from the explicit formulas for the action of the automorphisms $r_i$ in Proposition~\ref{prop-act-ri}. In fact, we have at once that $\varOmega_{a,b,c}$ is invariant by $r_i$, $\rb_i$ if $\{i,i+1\}\subset\{a,b,c\}$. We also have immediately that if $\{i,i+1\}\cap\{a,b,c\}=\{i\}$, then $\rb_i$ transforms $i$ into $i+1$, while if $\{i,i+1\}\cap\{a,b,c\}=\{i+1\}$, then $r_i$ transforms $i+1$ into $i$. So it remains to show
\begin{equation}\label{eq:actri-aw4}
r_3(\varOmega_{1,2,3})=\varOmega_{1,2,4} ,\qquad r_2(\varOmega_{1,2,4})=\varOmega_{1,3,4} ,\qquad r_1(\varOmega_{1,3,4})=\varOmega_{2,3,4} .
\end{equation}
First, we note that, either directly or using the relation $r_i\delta_i=\delta_i$, we have
\[r_1\varOmega_{12,3,4}=\varOmega_{12,3,4} ,\qquad r_2\varOmega_{1,23,4}=\varOmega_{1,23,4} ,\qquad r_3\varOmega_{1,2,34}=\varOmega_{1,2,34} .\]
Using the relation $r_3\delta_1=\delta_1r_2$ and $r_1\delta_3=\delta_3r_1$, we get also
\[r_3\varOmega_{12,3,4}=\varOmega_{12,3,4} ,\qquad r_1\varOmega_{1,2,34}=\varOmega_{1,2,34} .\]
Then, using~\eqref{eq:cas4}, we can write
\[\varOmega_{1,2,3}+\varOmega_{1,2,4}=\varOmega_{12,3,4}-\varOmega_{1,2,34}-\varOmega_{1,3,4}-\varOmega_{2,3,4} .\]
The right-hand side is invariant by $r_3$, so we deduce that the left-hand side is as well. Since we already know that $r_3\varOmega_{1,2,4}=\varOmega_{1,2,3}$ we conclude that $r_3\varOmega_{1,2,3}=\varOmega_{1,2,4}$. The other actions in~\eqref{eq:actri-aw4} are proved in a~similar way.
\end{proof}

\subsection[Central elements of aw(n) for any n]{Central elements of $\boldsymbol{\mathfrak{aw}(n)}$ for any $\boldsymbol{n}$}
The definition of $\omega_S$ introduced in the previous section for $n=4$, is now generalised. If ${S\subset\{1,\dots,n\}}$ we naturally consider the element $\omega_S$ defined below as an element of $\mathfrak{aw}(n')$ for any~$n'\geq n$.

\begin{prop}
Let $S\subset\{1,\dots,n\}$, with $|S|\geq3$. Let $I_1$, $I_2$, $I_3$ be three non-empty subsets of $S$ such that $I_1\cup I_2\cup I_3= S$ and $I_1<I_2<I_3$.
The quantity
\begin{equation}\label{eq:omegaS}
\omega_{S} = \sum_{\atopn{I\subset I_1, J\subset I_2}{K\subset I_3}} (-1)^{|S|-|I|-|J|-|K|} \varOmega_{I,J,K}
\end{equation}
is well defined in the sense that it does not depend on the choice of $I_1$, $I_2$, $I_3$.
\end{prop}
\begin{proof}
In the course of the proof, we will denote $\omega_{I_1, I_2, I_3}$ the right-hand side of~\eqref{eq:omegaS}.
We prove this proposition by recursion on the cardinality of $S$. The case $|S|=3$ is trivial. For $S=\{1,2,3,4\}$, the proposition is proven by equation~\eqref{eq:cas4}. Suitable applications of the maps $\rb_1,\dots ,\rb_{n-1}$ to the equalities~\eqref{eq:cas4} yield the corresponding equalities for any indices ${a<b<c<d}$. For example, $\rb_4$ changes 4 in 5.
This allows to prove the statement for $S=\{a,b,c,d\}$.

Now we take $k\geq5$, suppose that $\omega_S$ is well defined when $|S|\!<\!k$ and consider ${S\!=\!\{1,2,\dots ,k\}}$.
By recursion hypothesis, we have
\begin{equation*}
\omega_{\{1,\dots,k-1\}}=\omega_{\{1,\dots,p-1\},\{p,\dots,q-1\},\{q,\dots,k-1\}}=\omega_{\{1,\dots,p\},\{p+1,\dots,q-1\},\{q,\dots,k-1\}}.
\end{equation*}
From~\eqref{eq:deltaOmega}, the following formulas for the application of $\delta_1$ are easy to check:{\samepage
\begin{gather}
\delta_1\omega_{\{1,\dots,p-1\},\{p,\dots,q-1\},\{q,\dots,k-1\}} = \omega_{\{1,2,\dots,p\},\{p+1,\dots,q\},\{q+1,\dots,k\}}
\label{tata}\nonumber\\
\qquad{}+\omega_{\{1,3,\dots,p\},\{p+1,\dots,q\},\{q+1,\dots,k\}}+\omega_{\{2,\dots,p\},\{p+1,\dots,q\},\{q+1,\dots,k\}},\label{toto}\\
\delta_1\omega_{\{1,\dots,p\},\{p+1,\dots,q-1\},\{q,\dots ,k-1\}} = \omega_{\{1,2,\dots,p+1\},\{p+2,\dots,q\},\{q+1,\dots,k\}}
\nonumber\label{tutu}\\
\qquad{}+\omega_{\{1,3,\dots,p+1\},\{p+2,\dots,q\},\{q+1,\dots,k\}}+\omega_{\{2,\dots,p+1\},\{p+2,\dots,q\},\{q+1,\dots,k\}}\label{titi}.
\end{gather}}

\noindent
By recursion hypothesis, the two last terms in~\eqref{toto} and in~\eqref{titi} coincide, so that the two remaining terms coincide as well.
This shows that we can move a letter from $I_1$ to $I_2$. The similar statement from $I_2$ to $I_3$ is proven along the same lines.
This allows to relate any possible choices of $I_1$, $I_2$, $I_3$.
Suitable applications of the maps $\rb_1,\dots ,\rb_{n-1}$ yield the statement for any~$S$ of cardinal $k$, as it was done above for $k=4$.
\end{proof}

{\bf Action of the coproduct maps.} We deduce how the elements $\omega_S$ in $\mathfrak{aw}(n)$ are related to those of $\mathfrak{aw}(n+1)$ by the coproduct maps.
\begin{coro}\label{coro:deltaW}
Let $S'\subset\{1,\dots,n\}$ and $i\in\{1,\dots,n\}$.
\begin{itemize}\itemsep=0pt
\item If $i\notin S'$, we have
\begin{equation*}
\delta_i\omega_{S'}= \omega_{S} ,
\end{equation*}
where $S$ is obtained from $S'$ by increasing by 1 all elements greater than $i$.
\item If $i\in S'$, we have
\begin{equation}\label{eq:deltaWs}
\delta_i\omega_{S'}= \omega_{S} +\omega_{S\setminus\{i\}}+\omega_{S\setminus\{i+1\}},
\end{equation}
where $S$ is obtained from $S'$ by increasing by 1 all elements greater than $i$, and adding~${i+1}$.
\end{itemize}
\end{coro}
\begin{proof}
When $i\notin S'$, the formula is immediate using~\eqref{eq:deltaOmega} on any definition of $\omega_{S'}$. If $i\in S'$, from the preceding proposition, we can define $\omega_{S'}$ using a partition $(I_1,I_2,I_3)$ of $S'$ such that one of the subset is the singleton $\{i\}$. Then the formula~\eqref{eq:deltaWs} becomes straightforward to check using~\eqref{eq:deltaOmega}.
\end{proof}

{\bf Centrality of the elements $\boldsymbol{\omega_S}$.} Now we can prove that all elements $\omega_S$ are central. Note that this is equivalent to the centrality of the elements $\varOmega_{I_1,I_2,I_3}$ since both sets span the same space.
\begin{prop}
 For any $S\subset\{1,2,\dots ,n\}$ with $|S|\geq3$, the element $\omega_{S}$ is central in $\aw$.
\end{prop}
\begin{proof}
First, we note that it is enough to prove that $\omega_{1\dots k}$ is central for any $k\geq 3$. Indeed, if we have $S=\{i_1,\dots,i_k\}$ with $i_1<i_2<\dots <i_k$, there is a sequence of $\rb_i$'s sending $\omega_{1\dots k}$ to~$\omega_S$. For example, the sequence $\rb_{i_k}\dots\rb_{k}$ transforms the index $k$ in $\omega_{1\dots k}$ into $i_k$, and a similar argument allows to transform the other indices into $i_1,i_2,\dots,i_{k-1}$. So indeed the centrality of~$\omega_{1\dots k}$ implies the centrality of $\omega_S$ since they are related by an automorphism.

Now we prove that $\omega_{1\dots k}$ is central by induction on $k$. Let $k=3$. It is immediate that $\omega_{123}$ commutes with $C_I$ if $I$ is disjoint from $\{1,2,3\}$. So we take $I=\{a,\dots, b-1, b\}$ with~$a\leq 3$ and~$b\geq 3$ and we use induction on $b$. Propositions~\ref{propcas3} and~\ref{propcas4} deal with the cases~$b=3$ and~$b=4$. So let $b\geq 5$. In this case $C_I=\delta_{b-1}(C_{a\dots b-1})$. By induction hypothesis, $\omega_{123}$ commutes with $C_{a\dots b-1}$ and applying $\delta_{b-1}$, using that $\delta_{b-1}(\omega_{123})=\omega_{123}$, we get that $\omega_{123}$ commutes with $C_I$.

Next, let $k\geq 3$ and assume by induction that $\omega_S$ is central for any $S$ with $|S|\leq k$. We will prove that $\omega_{1\dots k+1}$ commutes with $C_I$ for any connected subset $I$ and it is enough to assume that $|I|>1$, since otherwise $C_I$ is central. Since $|I|>1$, then $C_I=\delta_i(C_J)$ for some $i$ and some~$J$. By induction, we have that $[\omega_{1\dots k},C_J]=0$, and applying $\delta_i$, we get
\[[\omega_{\{1,\dots, k+1\}}+\omega_{\{1,\dots, k+1\}/\{i\}}+\omega_{\{1,\dots, k+1\}/\{i+1\}} ,\ C_I]=0 ,\]
using Corollary~\ref{coro:deltaW} for the action of $\delta_i$ on $\omega_{1\dots k}$. By induction hypothesis, $C_I$ commutes with the two summands corresponding to subsets of size $k$, so we conclude that it commutes also with~$\omega_{\{1,\dots, k+1\}}$.
\end{proof}

{\bf Action of the automorphisms $\boldsymbol{r_1,\ldots,r_{n-1}}$.} The following result gives the action of the braid group automorphisms $r_1,\dots,r_{n-1}$ on the central elements $\omega_S$ and shows that this action simply amounts to the permutation action of the symmetric group on $n$ letters on subsets~$S$ of~$\{1,\dots,n\}$. Below, $(i,i+1)$ denotes the transposition of $i$ and $i+1$.

\begin{prop}\label{prop:romega}
For all $S\subset \{1,\dots,n\}$ with $|S|\geq 3$, and all $i=1,\dots,n-1$, we have
\begin{equation}\label{eq:romega}
r_i(\omega_S)=\rb_i(\omega_S)=\omega_{(i,i+1).(S)}.
\end{equation}
\end{prop}
\begin{proof}
We first consider $\varOmega_{I_1,I_2,I_3}$ as defined in~\eqref{eq:aw3cas}, and set $S=I_1\cup I_2\cup I_3$.

If $\{i,i+1\}\subset I_1$ (or $I_2$ or $I_3$) or if $\{i,i+1\}\cap S=\varnothing$, it is clear that $r_i(\varOmega_{I_1,I_2,I_3})=\varOmega_{I_1,I_2,I_3}$.
If~$i\in I_1$ and $i+1\not\in S$, then $\rb_i$ transforms $i$ into $i+1$ for each element $C_I$ entering the definition of
$\varOmega_{I_1,I_2,I_3}$, so that $\rb_i(\varOmega_{I_1,I_2,I_3})=\varOmega_{(i,i+1).I_1,I_2,I_3}$.
We get similar expressions when~$i\in I_2$ or~$i\in I_3$ and $i+1\not\in S$.
Reciprocally, if $i+1\in I_1$ and $i\not\in S$ we get $r_i(\varOmega_{I_1,I_2,I_3})=\varOmega_{(i,i+1).I_1,I_2,I_3}$ and analogous expressions for
$i+1\in I_2$ or $i+1\in I_3$ (and $i\not\in S$).

From the previous paragraph, we get immediately that
\begin{gather}
\text{when}\ i\in S \ \text{and} \ i+1\not\in S,\ \rb_i(\omega_S)=\omega_{(i,i+1).(S)}, \nonumber\\
\text{when}\ i+1\in S \ \text{and} \ i\not\in S,\ r_i(\omega_S)=\omega_{(i,i+1).(S)}.\label{ri-simple}
\end{gather}

Now, we are ready to prove~\eqref{eq:romega} using a recursion on $|S|$. Let $i\in\{1,\dots, n-1\}$, and consider $r_i(\omega_S)$.
We note that when $\{i,i+1\}\cap S=\varnothing$, the relation~\eqref{eq:romega} is obvious, so that we will always assume that $\{i,i+1\}\cap S$ is not empty.

We start with $|S|=3$. If $\{i,i+1\}\subset S$, then acting with suitable~$r_k$'s and~$\rb_k$'s commuting with~$r_i$, only using~\eqref{ri-simple}, it is sufficient to consider $S=\{i-1,i,i+1\}$, or $S=\{i,i+1,i+2\}$. Then, the result follows from Proposition~\ref{propcas3}.
If $i\in S$ and $i+1\not\in S$, then acting with suitable~$r_k$'s and~$\rb_k$'s commuting with $r_i$, only using~\eqref{ri-simple}, it is sufficient to consider $S=\{i-2,i-1,i\}$, $S=\{i-1,i,i+2\}$, or $S=\{i,i+2,i+3\}$. Then, the result follows from Proposition~\ref{propcas4}. The last case $i+1\in S$ and $i\not\in S$ is dealt with similar arguments.

Let $|S|>3$ and assume that the proposition is proved for sets of size smaller than $|S|$. Suppose $\{i,i+1\}\subset S$. Then, there exists a subset $S'$ with $3\leq |S'|<|S|$ such that using Corollary~\ref{coro:deltaW}, we have
\begin{equation*}
\omega_{S} = \delta_i\omega_{S'}-\omega_{S\setminus\{i\}}-\omega_{S\setminus\{i+1\}}.
\end{equation*}
The map $r_i$ exchanges $\omega_{S\setminus\{i\}}$ and $\omega_{S\setminus\{i+1\}}$
by induction hypothesis, and leaves $ \delta_i\omega_{S'}$ invariant, since $r_i\delta_i=\delta_i$.
We conclude that $r_i(\omega_{S}) =\omega_{S}$.

If $i\in S$ and $i+1\not\in S$, by~\eqref{ri-simple}, we only need to consider the action of $r_i$. Acting with suitable $r_k$'s and $\rb_k$'s commuting with $r_i$, only using~\eqref{ri-simple}, it is sufficient to consider
$S=S_1\cup S_2$ with $S_1=\{\dots,i\}$ and $S_2=\{i+2,\dots\}$ and both connected.
If $|S_2|>1$, there exists a subset~$S'$ with $3\leq |S'|<|S|$ such that using Corollary~\ref{coro:deltaW}, we have
\begin{equation*}
\omega_{S} = \delta_{i+2}\omega_{S'}-\omega_{S\setminus\{i+2\}}-\omega_{S\setminus\{i+3\}}.
\end{equation*}
Note that $r_i$ commutes with $\delta_{i+2}$. Using the induction hypothesis, we find
\begin{equation*}
r_i\omega_{S} = \delta_{i+2}\omega_{(i,i+1).S'}-\omega_{(i,i+1).S\setminus\{i+2\}}-\omega_{(i,i+1).S\setminus\{i+3\}}.
\end{equation*}
Using again the Corollary~\ref{coro:deltaW}, the right-hand side is indeed $\omega_{(i,i+1).S}$.

Finally, if $|S_2|\leq1$, then $|S_1|>2$, there exists a subset $S'$ with $3\leq |S'|<|S|$ such that using Corollary~\ref{coro:deltaW}, we have
\begin{equation*}
\omega_{S} = \delta_{i-2}\omega_{S'}-\omega_{S\setminus\{i-2\}}-\omega_{S\setminus\{i-1\}}.
\end{equation*}
Note that $r_i\delta_{i-2}=\delta_{i-2}r_{i-1}$. Using the induction hypothesis, we find
\begin{equation*}
r_i\omega_{S} = \delta_{i-2}\omega_{(i-1,i).S'}-\omega_{(i,i+1).S\setminus\{i-2\}}-\omega_{(i,i+1).S\setminus\{i-1\}}.
\end{equation*}
Using again the Corollary~\ref{coro:deltaW}, the right-hand side is indeed $\omega_{(i,i+1).S}$. The last case $i+1\in S$ and $i\not\in S$ is dealt with similar arguments.
\end{proof}

\begin{rema}\label{rema:omegaUp}
A consequence of the above properties of the central elements $\omega_S$ is that they are invariant under the map $(.)^{\rm up}$.
Indeed, it is checked by direct calculation for $\omega_{\{1,2,3\}}$, and then extended to $\omega_{\{1,\dots,k\}}$ using the fact that the coproduct maps commute with $(.)^{\rm up}$. Finally, the property for any set $S$ is obtained from the equality $r_i(X^{\rm up})=\rb_i(X)^{\rm up}$, valid for any element~$X$, and the fact that $r_i=\rb_i$
on the Casimir elements.
\end{rema}

{\bf Action of the automorphism $\boldsymbol{r_0}$.} So far we have proved in particular that the set
\begin{equation*}
\Gamma_n = \text{Span}\{\omega_S , S\subset\{1,\dots,n\}\}
\end{equation*}
of central elements is stable under the action of the automorphisms $r_1,\dots,r_{n-1}$ (and their inverses). Here we complete our study by considering the automorphism $r_0$. There are explicit formulas for the action of $r_0$ on the central elements $\omega_S$ but they are not very illuminating, so we shall be satisfied by proving that the automorphism $r_0$ does not produce new central elements

\begin{prop}
The automorphism $r_0$ of $\aw$ leaves stable $\Gamma_n$ and moreover satisfies $r_0^2(\omega_S)=\omega_S$ for any $S\subset\{1,\dots,n\}$.
\end{prop}
\begin{proof}
First, note that $r_0^2(\omega_S)=\omega_S$ follows from the formula $r_0^2=\rb_1\dots \rb_{n-2}\rb_{n-1}^2\rb_{n-2}\dots \rb_1$ of Proposition~\ref{prop-Delta} and the fact that $\Gamma_n$ is stable under $r_1,\dots,r_{n-1}$ and $r_1^2=\dots=r_{n-1}^2={\rm Id}$ on~$\Gamma_n$.

Then we prove by recurrence on $n$ that in $\aw$, we have that $r_0(\omega_{123})\in\Gamma_n$. For $n=3$, this is a straightforward calculation to show that $r_0(\omega_{123})=\omega_{123}$.

For $n=4$, a direct application on the definition of $\varOmega_{I_1,I_2,I_3}$ of the explicit formulas of Proposition~\ref{prop-act-ri} for the action of the automorphisms $r_a$ gives
\[r_0(\omega_{123})=\rb_3\rb_2\big(\varOmega_{12,3,4}^{\rm up}\big)=\rb_3\rb_2(\varOmega_{12,3,4}) ,\]
where we used the invariance under the map ${.}^{\rm up}$ of the elements $\varOmega_{I_1,I_2,I_3}$ for the last equality; this follows from Remark~\ref{rema:omegaUp}. The last expression is in $\Gamma_4$ since $\varOmega_{12,3,4}$ is in this span, which we already know to be stable by $\rb_2$, $\rb_3$.

Then we apply the coproduct map $\delta_4$ on the statement $r_0(\omega_{123})\in\Gamma_4$. The map $\delta_4$ commutes with $r_0$ (see Proposition~\ref{prop-r-delta}) and obviously sends $\omega_{123}$ to $\omega_{123}$. It also sends $\Gamma_n$ to $\Gamma_{n+1}$, from Corollary~\ref{coro:deltaW}, so by induction we conclude that $r_0(\omega_{123})\in\Gamma_n$ for any $n\geq 3$.

Next, quite similarly as in the previous paragraph, we apply $\delta_3$ on the statement ${r_0(\omega_{123})\!\in\!\Gamma_n}$. The map $\delta_3$ still commutes with $r_0$ and this allows to prove that $r_0(\omega_{1234})$ is in the desired span. Going on applying $\delta_4,\delta_5,\dots $, we get by induction that $r_0(\omega_{1\dots k})$ is in the desired span for all~$k\geq 3$.

Now, using Proposition~\ref{prop:romega}, one can go from $\omega_{1\dots k}$ to any $\omega_S$ with $|S|=k$ by applying a~suitable sequence of maps $r_i$. Moreover if $1\in S$, we can use only the maps $r_i$ with $i\geq 2$ (we do not need to touch the letter 1), which commute with $r_0$. Thus $r_0(\omega_S)$ is in the correct span if $1\in S$.

Finally, if $1\notin S$, then we trivially have $r_0(\omega_S)=\omega_S$ since $\omega_S$ is expressed in terms of generators $C_I$ with $1\notin I$, and these are all left invariant by $r_0$.
\end{proof}

\begin{exam}
During the proof, we got the equality $r_0(\omega_{1,2,3})=\rb_3\rb_2(\varOmega_{12,3,4})$. This can be used to give explicit expressions for the action of $r_0$. We skip the details and give the results for~${n=4}$. In the basis $\omega_{123}$, $\omega_{124}$, $\omega_{134}$, $\omega_{234}$, $\omega_{1234}$, the action of $r_0$ is given by the following matrix:
\[\left(\begin{array}{ccccc} 1 & 0 & 0 & 0 & \hphantom{-}0 \\
0 & 1 & 0 & 0 & \hphantom{-}0\\
0 & 0 & 1 & 0 & \hphantom{-}0\\
1 & 1 & 1 & 1 & -2\\
1 & 1 & 1 & 0 & -1\end{array}\right) .\]
Explicit but cumbersome formulas can be obtained for any $n$ by applying suitable coproduct maps and automorphisms as described during the proof.
\end{exam}

\begin{rema}
Since $r_0^2=r_1^2=\dots=r_{n-1}^2={\rm Id}$ on the set of central elements $\omega_S$, the action of the braid group on $n+1$ strands by automorphisms become an action of the symmetric group on $n+1$ letters on $\Gamma_n$.
\end{rema}

\section[Connections with U\_q(sl\_2)\^{}\{otimes n\} and the skein algebra]{Connections with $\boldsymbol{{\rm U}_q(\mathfrak{sl}_2)^{\otimes n}}$ and the skein algebra}\label{sec-connections}
In this section, we show how the $\aw$ algebra is related to ${\rm U}_q(\mathfrak{sl}_2)^{\otimes n}$ and to a certain Kauffman bracket skein algebra.
More precisely, one has to consider in ${\rm U}_q(\mathfrak{sl}_2)^{\otimes n}$ the subalgebra generated by the intermediate Casimir elements. We note that this algebra was recently shown to be isomorphic to the Kauffman bracket skein algebra of the sphere with $n+1$ punctures~\cite{CL}.
Therefore, we will mostly discuss the ${\rm U}_q(\mathfrak{sl}_2)^{\otimes n}$ case.

Consider the quantum group ${\rm U}_q(\mathfrak{sl}_2)$ following the conventions and notations of~\cite{avatar}, and denote by $Q$ its Casimir element (note that we have divided the Casimir element of~\cite{avatar} by $q+q^{-1}$, as explained in Example~\ref{ex:aw3}). By repeated application of the coproduct map of ${\rm U}_q(\mathfrak{sl}_2)$ on $Q$,
we construct the intermediate Casimir element $Q_I$, where $I$ is a connected subset of $\{1,\dots,n\}$.

We have a morphism of algebras given by
\begin{align*}%\label{eq:aw-uq}
\vph\colon\
\aw &\to {\rm U}_q(\mathfrak{sl}_2)^{\otimes n},\\
C_I &\mapsto Q_I,
\end{align*}
It is known that the defining relation of $\aw$ are obeyed by the elements $Q_I$~\cite{CL, DeC}. Alternatively, the statement for general $n$
is obtained by applying the coproduct maps and the braid group automorphism on
the corresponding statement for $n=3$. We refer to the discussion in Section~\ref{sect:aw3-n}.
\begin{rema}
We wish to stress that several authors~\cite{CL, DeC} call the image of the map $\vph$ the Askey--Wilson algebra $\mathfrak{aw}(n)$. We follow instead the terminology of~\cite{avatar}, where the image of $\vph$ was referred to as the special Askey--Wilson algebra.
\end{rema}
The maps $\delta_i$ and $r_i$, $\rb_i$ are interpreted naturally through the morphism $\vph$. Indeed, $\delta_i$ is sent to the coproduct map
\begin{equation*}
\Delta_i\colon \ {\rm U}_q(\mathfrak{sl}_2)^{\otimes n}\ \to \ {\rm U}_q(\mathfrak{sl}_2)^{\otimes (n+1)},
\end{equation*}
which is the coproduct $\Delta$ applied in the $i^{\rm th}$ copy of ${\rm U}_q(\mathfrak{sl}_2)$. In the same way, the morphisms $r_i$, $i=1,\dots,n-1$, are sent to the morphisms
\begin{align*}
\rho_i\colon \ {\rm U}_q(\mathfrak{sl}_2)^{\otimes n} &\to {\rm U}_q(\mathfrak{sl}_2)^{\otimes n},\\
 X & \mapsto \tau_{i,i+1}\big( R_{i,i+1} X R_{i,i+1}^{-1}\big),
\end{align*}
where $R$ is the $R$-matrix and $\tau_{i,i+1}$ is the flip operation between the $i^{\rm th}$ and $(i+1)^{\rm th}$ copies of ${\rm U}_q(\mathfrak{sl}_2)$. This was shown for $n=3$ in~\cite{CGVZ}. It is then extended to general $n$ using the relations between the coproduct maps $\delta_i$ and the automorphisms $r_i$ in Proposition~\ref{prop-r-delta}, see also the discussion in Section~\ref{sect:aw3-n}.

Similarly, the maps $\delta_i$ and $r_i$, $\rb_i$ are also interpreted in the skein algebra by, respectively, an operation doubling the puncture $i$, and the half Dehn twist, see~\cite{avatar} or~\cite{CL} for more details.
\begin{rema}
The map $r_0$ is naturally interpreted in the skein algebra as a half Dehn twist involving the puncture at infinity. Such an easy interpretation of the map $r_0$ in ${\rm U}_q(\mathfrak{sl}_2)^{\otimes n}$ does not seem to exist.
\end{rema}

To describe the algebra generated by the intermediate Casimir elements in ${\rm U}_q(\mathfrak{sl}_2)^{\otimes n}$, we need to determine the kernel of $\vph$. A first step in this determination is given by the following proposition, which becomes an easy by-product of our construction.
\begin{prop}
All the central elements $\omega_S$, $S\subset\{1,\dots,n\}$ are in the kernel of the map $\vph$:%
\begin{equation*}
\vph(\omega_S)=0.
\end{equation*}
\end{prop}
\begin{proof}
The fact that $\vph(\omega_{\{1,2,3\}})=0$ is known, see, e.g.,~\cite{avatar}. The general statement then follows from the fact that any $\omega_S$ can be obtained from $\omega_{\{1,2,3\}}$ through suitable applications of the coproduct maps $\delta_i$ and the automorphisms $r_i$, $\rb_i$, see Corollary~\ref{coro:deltaW} and Proposition~\ref{prop:romega}.
\end{proof}

We do not know a set of generators for the kernel of $\vph$ in general. For $n=3$, it is known that $\omega_{\{1,2,3\}}$ generates the kernel~\cite{avatar}.
For $n=4$, a generating set was given in~\cite[Appendix~A]{CL}. We were able to check
 that all the expressions\footnote{By expression, we mean here the right-hand side
minus the left-hand side of each relation.} in~\cite[Appendix~A]{CL} correspond to central elements in $\mathfrak{aw}(4)$, so that the kernel of $\vph$ is generated by central elements for $n=4$. Moreover, we also showed that the so-called \textit{loop triple relations}, \textit{link triple relations} and \textit{double and triple crossing relations} in~\cite{CL} are identically zero in the algebra
$\mathfrak{aw}(4)$, but we still do not know if the kernel is generated by the central elements $\omega_S$. Similarly, for $n\geq5$, it would be interesting too compare~$\aw$ with the defining relations of the Kauffman bracket skein algebra given in~\cite{Skein5}.

\section{Limit to the Racah algebra}\label{sec-limit}

In this section, we consider the limit $q\to1$ in the relations of the algebra $\mathfrak{aw}(n)$, and show how to recover some relations of the Racah algebra $R(n)$. In particular, we find all the defining relations of $R(4)$, as described in the appendix of~\cite{CFR}.

We use the following change of generators, for any connected subset $I$,
\begin{equation*}
C_I = \frac{\big(q-q^{-1}\big)^2}{q+q^{-1}} K_I+1.
\end{equation*}
For any expression in $\aw$ in terms of the generators $C_I$, we replace $C_I$ using the above change of generators, and then look at the first non-trivial coefficient in the expansion around $q=1$.
\subsection{Relations with 3 subsets}
Let $I_1$, $I_2$, $I_3$ be a monotonic sequence of adjacent subsets.
The first non-trivial coefficients in relations~\eqref{relaw31v} and~\eqref{relaw32} in the limit $q\to1$ give
\begin{subequations}
\begin{align}
\label{eq:Rac1}
\frac{1}{2}\big[K_{I_1I_2}, [K_{I_1I_2},K_{I_2I_3}] \big] ={}& K_{I_1I_2}^2 + \{K_{I_1I_2},K_{I_2I_3}\} - (K_{I_1}+K_{I_2}+K_{I_3}+K_{I_1I_2I_3})K_{I_1I_2} \nonumber\\
&{}- (K_{I_1}-K_{I_2})(K_{I_3}-K_{I_1I_2I_3}) , \\
\label{eq:Rac2}
\frac{1}{2}\big[K_{I_2I_3}, [K_{I_2I_3},K_{I_1I_2}] \big] ={}& K_{I_2I_3}^2 + \{K_{I_1I_2},K_{I_2I_3}\} - (K_{I_1}+K_{I_2}+K_{I_3}+K_{I_1I_2I_3})K_{I_2I_3} \nonumber\\
&{}- (K_{I_1}-K_{I_1I_2I_3})(K_{I_3}-K_{I_2}) ,
\end{align}
\end{subequations}
where $\{A,B\}=AB+BA$ is the anti-commutator.
One recognises the relations defining an algebra $R(I_1,I_2,I_3)$ isomorphic to the Racah algebra $R(3)$, see, e.g.,~\cite{avatar}, where such limit was already considered for $n=3$.

\subsection{Relations with 4 subsets}
Next, we consider a monotonic sequence of adjacent subsets $I_1$, $I_2$, $I_3$, $I_4$, and the relations~\eqref{relaw4a}--\eqref{relaw4e}. We use the notation with a prime on the number of the equation
when we consider it by exchanging the order of the subsets $(I_1,I_2,I_3,I_4)\to(I_4,I_3,I_2,I_1)$.
It is easy to see that all the relations~\eqref{relaw4a}--\eqref{relaw4e} as well as their prime versions
have the same first non-trivial coefficient in the limit $q\to 1$:
\begin{gather}\label{eq:cub0}
 [K_{I_1I_2},K_{I_2I_3}]\! + \![K_{I_2I_3},K_{I_3I_4}]\! -\! [K_{I_1I_2I_3},K_{I_3I_4}]\! - \![K_{I_1I_2},K_{I_2I_3I_4}]\! + \![K_{I_1I_2I_3},K_{I_2I_3I_4}]\! = \!0,
\end{gather}
One recognises one of the relations defining an algebra $R(I_1,I_2,I_3,I_4)$ isomorphic to the Racah algebra $R(4)$:
it corresponds to the relation (A.13) of~\cite{CFR}.

Now, considering the sums~\eqref{relaw41}+\eqref{relaw43}$'$,~\eqref{relaw41}$'$+\eqref{relaw43}, \eqref{relaw42}+\eqref{relaw44}$'$,
\eqref{relaw42}$'$+ \eqref{relaw44},
\eqref{relaw45}+\eqref{relaw46}$'$,~\eqref{relaw45}$'$+\eqref{relaw46} leads to six new relations.
For example, \eqref{relaw41}+ \eqref{relaw43}$'$ gives
%\begin{subequations}
 \begin{gather*}%\label{eq:cub1}
\frac12\big[K_{I_3I_4} , [K_{I_1I_2},K_{I_2I_3}]\big] =
K_{I_1I_2}(K_{I_2I_3}+K_{I_3I_4}-K_{I_2I_3I_4}-K_{I_3}) +K_{I_2I_3}(K_{I_3I_4}-K_{I_1I_2I_3I_4})
\\
\hphantom{\frac12\big[K_{I_3I_4} , [K_{I_1I_2},K_{I_2I_3}]\big] =}{}-K_{I_3I_4} K_{I_2}
+K_{I_1I_2I_3}(K_{I_2I_3I_4}-K_{I_3I_4}-K_{I_2})-K_{I_2I_3I_4} K_{I_3}
\\
\hphantom{\frac12\big[K_{I_3I_4} , [K_{I_1I_2},K_{I_2I_3}]\big] =}{}
+(K_{I_2}+K_{I_3}) K_{I_1I_2I_3I_4}+K_{I_2} K_{I_3} ,
\end{gather*}
where we have used~\eqref{eq:cub0} to simplify the right-hand side of the relation.
Among these six relations, four of them
are (A.14)--(A.17) in the notations of~\cite{CFR}, where the indices $1$, $2$, $3$, $4$ are replaced by the subsets $I_1$, $I_2$, $I_3$, $I_4$. The remaining ones can be obtained using these four relations and the relation~\eqref{eq:cub0}.

The four relations, together with~\eqref{eq:Rac1}--\eqref{eq:Rac2} and~\eqref{eq:cub0}, form a complete set of defining relations for the algebra $R(I_1,I_2,I_3,I_4)$ isomorphic to the Racah algebra $R(4)$, see \cite[Appendix~A]{CFR} (it can be shown that the remaining defining relations
in~\cite[Appendix~A]{CFR} are consequences for these ones).

The relations obtained above involving 3 and 4 subsets of $\{1,\dots,n\}$ are a set of defining relations of the so-called higher rank Racah algebra $R(n)$ as studied, e.g., in~\cite{CGPV}. In fact, it is easy to show that the defining relations (3.2f) and (3.2g) in~\cite{CGPV} are a consequence of the others.

\subsection{Casimir elements}
The first non-trivial coefficient in the expression~\eqref{eq:aw3cas} of $\varOmega_{1,2,3}$ provides the Casimir element~$w_{123}$ of~$R(3)$, see~\cite{avatar}.
Similarly, for $\varOmega_{I_1,I_2,I_3}$, it provides the Casimir element of the algebra $R(I_1,I_2,I_3)$.

 Up to a global multiplicative constant, using the calculation made in~\cite{CFR}, it follows that the central element $\omega_{\{1,2,3,4\}}$ provides in the limit the central element of $R(n)$ called $x_{1234}$ in~\cite{CFR, CGPV}.

 We note that in the Racah algebra $R(n)$ the natural way to build central elements involving~5 and~6 indices leads to elements which were shown to be equal to zero in the algebra $R(n)$~\cite{CGPV}. We do not know if and how this fact has its counterpart on our central elements $\omega_S$ of $\aw$.

\appendix

\section{Details for some proofs}\label{sec-appendix}

All along the proofs, for simplicity, we will use a notation which omits the accolades and the symbol for union of sets, for example, as
\[I\backslash a, \quad \text{for}\ I\backslash\{a\}\qquad \text{and}\qquad Ia\ \text{or}\ aI, \quad \text{for}\ I\cup\{a\} .\]
The last one, using $Ia$ or $aI$, will not lead to any ambiguity since it will be used only when $a$ is adjacent to $I$ so that $C_{Ia}=C_{aI}$ with our conventions.

\subsection[Proof of the morphism property for r\_i, bar r\_i, i=1,...,n-1]{Proof of the morphism property for $\boldsymbol{r_i}$, $\boldsymbol{\rb_i}$, $\boldsymbol{i=1,\ldots,n-1}$}
The morphism property is what remains to be done for the proof of Theorem~\ref{theo-matricesR}. We need to check that the maps $r_i$, $\rb_i$ for $i=1,\dots,n-1$, given on the generators by (\ref{def-ri})--(\ref{def-rbi}) preserve the defining relations (\ref{relcommv})--(\ref{relaw41v}) of $\aw$. We will use Proposition~\ref{prop-ri} without mentioning.

As a first reduction, recall that $\rb_i(X^{\rm up})=(r_i(X))^{\rm up}$. Therefore, if we prove that the maps $r_i$, $i=1,\dots,n-1$ are morphisms, it will imply that the maps $\rb_i$, $i=1,\dots,n-1$, are morphisms as well. So we only need to deal with the maps $r_i$.

{\bf Commutation relations.} We start with the commutation relations (\ref{relcommv}). We take~$I$ and~$J$ as in~(\ref{relcommv}). There are several cases to consider.
\begin{itemize}\itemsep=0pt
\item Assume that $r_i(C_I)=C_I$ and $r_i(C_J)=C_J$. Then there is nothing to do.

\item Assume that $r_i(C_I)=C_I$ and $r_i(C_J)\neq C_J$ (the case with $I$ and $J$ exchanged is similar). We use the action of $r_i$ given in Proposition~\ref{prop-ri}. With notations such that $\{a\}=J\cap \{i,i+1\}$ and $\{a,b\}=\{i,i+1\}$, we must check that
\[[C_I, -[C_{i,i+1},C_{J}]_q+C_bC_{J\backslash a}+C_{a}C_{Jb}]=0 .\]
This is true since $C_I$ commutes, using the defining commutation relations, with all terms appearing in the expression for $r_i(C_J)$.

\item Assume that $r_i(C_I)\neq C_I$ and $r_i(C_J)\neq C_J$ and $I\cap J=\varnothing$. This happens when
\[\ \ \ldots \bullet,\underbrace{\bullet,\dots,\bullet}_{I_1},a,b,\underbrace{\bullet,\dots,\bullet}_{I_4},\bullet,\ldots ,\]
with $\{a,b\}=\{i,i+1\}$ and $I=I_1a$ and $J=bI_4$. Applying $r_i$ on $[C_I,C_J]=0$ gives the relation
\[[C_{bI_1},C_{aI_4}]=0 .\]
This is a particular case of relation~\eqref{rel:coma1} proved in Lemma~\ref{lemm:relcom2}.

\item Finally, assume that $r_i(C_I)\neq C_I$ and $r_i(C_J)\neq C_J$ and $I\subset J=\varnothing$. This happens when
\[\ \ \ldots \bullet,a,b,\underbrace{\bullet,\dots,\bullet}_{I_3},\underbrace{\bullet,\dots,\bullet}_{I_4},\bullet,\ldots ,\]
with $\{a,b\}=\{i,i+1\}$ and $I=bI_3$ and $J=bI_3I_4$. Applying $r_i$ on $[C_I,C_J]=0$ gives the relation:
\[[C_{aI_3},C_{aI_3I_4}]=0 .\]
This is a particular case of relation~\eqref{rel:coma3}, proved in Lemma~\ref{lemm:relcom2}.
\end{itemize}

{\bf The remaining relations.} These relations are of the form
\begin{equation}\label{app-relKIJ}
C_{K}=-[C_{I},C_{J}]_q+C_{I\backslash J}C_{J\backslash I}+C_{I\cap J}C_{I\cup J} .
\end{equation}
We have three different situations which may happen when acting with $r_i$. Consider the two elements inside the $q$-commutator:
\begin{itemize}\itemsep=0pt
\item The action of $r_i$ leaves both of them invariant. This happens when $\{i,i+1\}$ is disjoint from every occurring subset. In this case, $r_i$ leaves stable every elements appearing in the relation, and the relation is trivially preserved.
\item The action of $r_i$ only leaves one of them invariant. We will deal with these cases just below.
\item Finally, the action of $r_i$ is non-trivial on both elements. In this case, one can check that $r_i$ leaves invariant the left hand side. These are the difficult cases. We will deal with them one by one.
\end{itemize}
We treat the second situation. We must apply $r_i$ on a relation~\eqref{app-relKIJ}
and we assume that $r_i(C_J)=C_J$ while $r_i$ acts non-trivially on $C_I$ (the other case is completely similar). There are two possible situations: $\{i,i+1\}\cap J=\varnothing$ or $\{i,i+1\}\subset J$. We will give the details for the first situation, the other one can be treated in a similar way.

Let $\{a\}=\{i,i+1\}\cap I$ and $\{a,b\}=\{i,i+1\}$. We apply $r_i$ on the right hand side of the relation and, according to the action of $r_i$ given in Proposition~\ref{prop-ri}, we find
\begin{gather*}
 -[r_i(C_{I}),C_{J}]_q+r_i(C_{I\backslash J})C_{J\backslash I}+C_{I\cap J}r_i(C_{I\cup J})\\
 \qquad = -[C_{i,i+1},-[C_{I},C_{J}]_q+C_{I\backslash J}C_{J\backslash I}+C_{I\cap J}C_{I\cup J}]_q + XX \\
 \qquad = -[C_{i,i+1},C_K]_q + XX .
\end{gather*}
For the first equality, we use the explicit action of $r_i$ and we collect the various $q$-commutators appearing, using $q$-Jacobi. The remaining terms are collected in
\[
XX=-[C_bC_{I\backslash a}+C_aC_{Ib},C_J]_q+\big(C_{I\backslash Ja}C_b+C_aC_{Ib\backslash J}\big)C_{I\backslash J}+C_{I\cap J}\big(C_bC_{IJ\backslash a}+C_aC_{IJb}\big) .\]
In every cases fitting this situation, it is straightforward to check that
\begin{gather*}
-\big[C_{I\backslash a},C_J\big]_q=C_{K\backslash a}+C_{I\backslash aJ}C_{J\backslash I}+C_{I\cap J}C_{IJ\backslash a} ,\\
-[C_{Ib},C_J]_q=C_{Kb}+C_{Ib\backslash J}C_{J\backslash I}+C_{I\cap J}C_{IbJ} ,
\end{gather*}
and we find that $XX$ is equal to $C_{K\backslash a}C_b+C_aC_{Kb}$. Comparing with the action of $r_i$ on the left hand side of~\eqref{app-relKIJ}, we get
$r_i(C_K) = -[C_{i,i+1},C_K]_q +C_{K\backslash a}C_b+C_aC_{Kb}$, which is true (Proposition~\ref{prop-ri}).

{\bf The remaining case of relation (\ref{relaw41v}).} From now on, we need only to treat the cases where~$r_i$ acts non-trivially on both sides of the $q$-commutator. For this relation, this means that we have
\begin{equation}\label{Case41} \ldots \bullet,\underbrace{\bullet,\dots,\bullet}_{I_1},\underbrace{\bullet,\dots,\bullet,a}_{I_2},\underbrace{b,\bullet,\dots,\bullet}_{I_3},\underbrace{\bullet,\dots,\bullet}_{I_4},\bullet,\ldots ,\qquad \text{where $\{a,b\}=\{i,i+1\}$.}\end{equation}
We start with some preparations. For two disjoint connected subsets $I_1$ and $I_4$, consider the set of relations
\begin{equation}\label{Rel-I1-I3}C_{I_1I_4}=-[C_{I_1H},C_{HI_4}]_q+C_{I_1}C_{I_4}+C_{H}C_{I_1HI_4}
\end{equation}
for any connected subset $H$ between $I_1$ and $I_4$ adjacent to either $I_1$ or $I_4$. These relations are satisfied, due to the defining relation (\ref{relaw41v}) and relation~\eqref{relaw43}. Denote $\operatorname{dist}(I_1,I_4)$ the size of the hole between $I_1$ and $I_4$.
\begin{prop}\label{prop-equiv}
For given $I_1$ and $I_4$, the relations in the set \eqref{Rel-I1-I3} are all equivalent modulo the commutation relations and the relations of the type \eqref{Rel-I1-I3} for subsets $I'_1$ and $I'_4$ with $\operatorname{dist}(I'_1,I'_4)<\operatorname{dist}(I_1,I_4)$.
\end{prop}
\begin{proof}
Assume that $H$ is adjacent to $I_1$. Split the hole between $H$ and $I_4$ into $H_1$ and $H_2$ according to the following picture:
\[\ \ \ldots \bullet,\underbrace{\bullet,\dots,\bullet}_{I_1},\underbrace{\bullet,\dots,\bullet}_{H},\underbrace{\bullet,\dots,\bullet}_{H_1},\underbrace{\bullet,\dots,\bullet}_{H_2},\underbrace{\bullet,\dots,\bullet}_{I_4},\bullet,\ldots .\]
Then in the right hand side of (\ref{Rel-I1-I3}), replace $C_{HI_4}$ by its expression starting with $[C_{HH_2},C_{H_2I_4}]_q$, obtained from (\ref{Rel-I1-I3}) with $I_1$ replaced by $H$. Using the $q$-Jacobi relation and also relation (\ref{Rel-I1-I3}) with $I_4$ replaced by $H_2$, we find that the relation (\ref{Rel-I1-I3}) becomes
\[C_{I_1I_4}=-[C_{I_1H_2},C_{H_2I_4}]_q+C_{I_1}C_{I_4}+C_{H_2}C_{I_1H_2I_4} .\]
This proves the proposition, since $H_2$ was arbitrarily chosen (adjacent to $I_4$).
\end{proof}

Now we come back to the proof for the defining relation (\ref{relaw41v}) in the situation (\ref{Case41}), and we reason by induction on $\operatorname{dist}(I_1,I_4)$ (which is at least 2). If the hole between $I_1$ and $I_4$ contains strictly more than two elements, thanks to Proposition~\ref{prop-equiv} and using the induction hypothesis, we can replace relation (\ref{relaw41v}) by the one where $I_3$ is replaced by a subset $H$ either adjacent to~$I_1$ or to $I_4$ and such that $H\cap\{i,i+1\}=\varnothing$. This relation is trivially preserved by $r_i$.

Finally assume that $\operatorname{dist}(I_1,I_4)=2$, so that $I_2=\{a\}$ and $I_3=\{b\}$. The action of $r_i$ on the relation gives
\[C_{I_1I_4}=-[C_{I_1a},C_{aI_4}]_q+C_{I_1}C_{I_4}+C_{a}C_{I_1aI_4} ,\]
where we have used Proposition~\ref{prop-act-ri} for the action of $r_i$. This relation is proved in~\eqref{relaw43}.

{\bf The remaining cases of relation (\ref{relaw31v}).} Recall that the relation is
\begin{equation}\label{relaw31v-app}
C_{I_1I_2}=-[C_{I_2I_3},C_{I_1I_3}]_q+C_{I_1}C_{I_2}+C_{I_3}C_{I_1I_2I_3} .
\end{equation}
Using the definition of $C_{I_1I_3}$ and the $q$-Jacobi relation, we find the equivalent form of this relation:
\begin{equation}\label{relaw31v-app-eq}
C_{I_1I_2}=-[C_{I_3I_1},C_{I_2I_3}]_q+C_{I_1}C_{I_2}+C_{I_3}C_{I_1I_2I_3} .
\end{equation}
Proving that (\ref{relaw31v-app}) is preserved by $r_i$ is equivalent to proving that (\ref{relaw31v-app-eq}) is preserved by $r_i$, since we only used the commutation relations (which are already proven to be preserved by $r_i$) to move from one to the other.

Recall that it remains only to treat the case where $r_i$ acts non-trivially on both sides of the $q$-commutator in (\ref{relaw31v-app}). This happens when $\{i,i+1\}$ is either $\{a,b\}$ or $\{c,d\}$ as pictured below
\[ \ldots \bullet,\underbrace{\bullet,\dots,\bullet,a}_{I_1},\underbrace{b,\bullet,\dots,\bullet}_{I_2},\underbrace{\bullet,\bullet,\dots,c}_{I_3},d,\bullet,\ldots.\]

\textbf{Case 1: $\boldsymbol{\{i,i+1\}=\{a,b\}}$.} We are going to reason by induction on the size of $I_1$. So first, assume that $I_1=\{a\}$. We denote $I_2=bI'_2$. Using the equivalent form (\ref{relaw31v-app-eq}), the action of $r_i$ gives, according to Proposition~\ref{prop-act-ri}:
\[C_{abI'_2}=-[C_{I_3b},C_{aI'_2I_3}]_q+C_{b}C_{aI'_2}+C_{I_3}C_{abI'_2I_3} .\]
Such a relation was proven in Lemma~\ref{lemm:relcom2}, relation~\eqref{eq:2h5}.

Now assume that we can split $I_1$ into a union of two non-empty connected subsets $I_0\cup I'_1$ with $I'_1$ adjacent to $I_2$. Relation (\ref{relaw31v-app}) reads now as
\begin{equation}\label{relaw31v-app2}C_{I_0I'_1I_2}=-[C_{I_2I_3},C_{I_0I'_1I_3}]_q+C_{I_0I'_1}C_{I_2}+C_{I_3}C_{I_0I'_1I_2I_3} .
\end{equation}
Now we are going to use the following relations:
 \begin{gather}
C_{I_0I'_1I_3} =-[C_{I_2I_0},C_{I'_1I_2I_3}]_q+C_{I_0}C_{I'_1I_3}+C_{I_2}C_{I_0I'_1I_2I_3}, \label{app-inter11} \\
C_{I_3I_0} =-[C_{I_3I_2},C_{I_2I_0}]_q+C_{I_3}C_{I_0}+C_{I_2}C_{I_3I_2I_0},\label{app-inter12} \\
C_{I_0I'_1} =-[C_{I_3I_2I_0},C_{I'_1I_2I_3}]_q+C_{I_0}C_{I'_1}+C_{I_2I_3}C_{I_0I'_1I_2I_3},\label{app-inter13}\\
C_{I_0I'_1I_2} =-[C_{I_3I_0}, C_{I_1I_2I_3}]_q +C_{I_0}C_{I'_1I_2}+C_{I_3}C_{I_0I'_1I_2I_3},\label{app-inter14}\\
C_{I'_1I_2} =-[C_{I_2I_3},C_{I'_1I_3}]_q+C_{I'_1}C_{I_2}+C_{I_3}C_{I'_1I_2I_3}.\label{app-inter15}
 \end{gather}
We use~\eqref{app-inter11} to replace $C_{I_0I'_1I_3}$ in the $q$-commutator in (\ref{relaw31v-app2}), and use all the others to calculate the resulting expression, and we find that the relation simply becomes
\begin{equation}\label{relaw31v-app3}C_{I_0I'_1I_2}=-[C_{I_3I_0},C_{I'_1I_2I_3}]_q+C_{I_0}C_{I'_1I_2}+C_{I_3}C_{I_0I'_1I_2I_3} .
\end{equation}
It is easy to check that the relations~\eqref{app-inter11}--\eqref{app-inter15} we have been using are all valid (they are either defining relations, or were proved in Lemma~\ref{prop:def-nico}). For relations~\eqref{app-inter11}--\eqref{app-inter14}, the map $r_i$ leaves at least one member of the $q$-commutator invariant, and thus these relations are preserved by $r_i$ (see the beginning of the proof of the remaining relations above). The stability of~\eqref{app-inter15} by $r_i$ is the induction hypothesis since this is indeed (\ref{relaw31v-app}) with $I_1$ replaced by $I'_1$. Therefore, we have been using relations that we know are preserved by $r_i$ to transform~(\ref{relaw31v-app2}) into the relation~(\ref{relaw31v-app3}). This latter relation is trivially preserved by~$r_i$. We conclude that the relation~(\ref{relaw31v-app2}) that we started with is preserved by $r_i$.

\textbf{Case 2: $\boldsymbol{\{i,i+1\}=\{c,d\}}$.} We are going to reason by induction on the size of $I_3$. So first, assume that $I_3=\{c\}$. Using the equivalent form (\ref{relaw31v-app-eq}), the action of $r_i$ gives, according to Proposition~\ref{prop-act-ri} as
\[C_{I_1I_2}=-[C_{dI_1},C_{dI_2}]_q+C_{I_1}C_{I_2}+C_{d}C_{dI_2I_1} .\]
Such a relation was proven in Lemma~\ref{lemm:relcom2}, relation~\eqref{eq:2h2} (with $(I_1,I_2,I_3,I_4)\to (d,c,I_2,I_1)$).

Now assume that we can split $I_3$ into a union of two non-empty connected subsets $I'_3\cup I_4$ with $I'_3$ adjacent to $I_2$. Relation (\ref{relaw31v-app}) reads now as
\begin{equation}\label{relaw31v-app4}C_{I_1I_2}=-[C_{I_2I'_3I_4},C_{I_1I'_3I_4}]_q+C_{I_2}C_{I_1}+C_{I'_3I_4}C_{I_1I_2I'_3I_4} .
\end{equation}
Now we are going to use the following relations:
 \begin{gather}
C_{I_1I'_3I_4} =-[C_{I_1I_2I_4},C_{I_2I'_3}]_q+C_{I_1I_4}C_{I'_3}+C_{I_2}C_{I_1I_2I'_3I_4} , \label{app-inter21} \\
C_{I'_3I_1} =-[C_{I_2I'_3I_4},C_{I_1I_2I_4}]_q+C_{I_1}C_{I'_3}+C_{I_2I_4}C_{I_1I_2I'_3I_4} ,\label{app-inter22} \\
C_{I'_3I_4} =-[C_{I_2I_4},C_{I_2I'_3}]_q+C_{I_4}C_{I'_3}+C_{I_2}C_{I_2I'_3I_4} ,\label{app-inter23}\\
C_{I_1I_2I'_3} =-[C_{I_2I'_3I_4}, C_{I_1I_4}]_q +C_{I_2I'_3}C_{I_1}+C_{I_4}C_{I_1I_2I'_3I_4} .\label{app-inter24}
 \end{gather}
We use~\eqref{app-inter21} to replace $C_{I_0I'_1I_3}$ in the $q$-commutator in (\ref{relaw31v-app4}), and use all the others to calculate the resulting expression, and we find that the relation simply becomes
\begin{equation*}%\label{relaw31v-app5}
C_{I_1I_2}=-[C_{I'_3I_1},C_{I_2I'_3}]_q+C_{I_1}C_{I_2}+C_{I'_3}C_{I_1I_2I'_3} .
\end{equation*}
Now this final relation is trivially preserved by $r_i$. Moreover, it is easy to check that the relations~\eqref{app-inter21}--\eqref{app-inter24} we have been using are all valid (they are either defining relations, or were proved in Lemma~\ref{prop:def-nico}). For relations~\eqref{app-inter21}--\eqref{app-inter23}, the map $r_i$ leaves at least one member of the $q$-commutator invariant, and thus these relations are preserved by $r_i$ (see the beginning of the proof of the remaining relations above). Finally, the stability of~\eqref{app-inter24} is the induction hypothesis since this is indeed (\ref{relaw31v-app}) with $I_3$ replaced by $I_4$ (and $I_2$ replaced by $I_2I'_3$). So, as in Case 1, this concludes the verification.

\subsection[Proof of Proposition~\ref{prop-act-ri} for r\_0]{Proof of Proposition~\ref{prop-act-ri} for $\boldsymbol{r_0}$}
The formula to be proven is $r_0(C_{I_1I_2\dots I_k})=C_{H_kH_{k-1}\dots H_1\, 1}$ when $1\in I_1$ and $I_1<I_2<\dots I_k$ is an increasing sequence of connected subsets, and $H_1<\dots <H_k$ is the complementary sequence in~$\{1,\dots,n\}$.

We use induction on $k$. For $k=1$, this is just the definition of $r_0$. For $k\geq 2$, we use the definition~\eqref{def-CIgen} of $C_{I_1I_2\dots I_k}$ using the hole $H_1$ between $I_1$ and $I_2$, and using the induction hypothesis for the action of $r_0$, we are led to proving
\begin{gather}\label{app-r0-k}
C_{H_kH_{k-1}\dots H_1\, 1}=-[C_{H_kI_kH_{k-1}\dots H_2I_2\, 1},C_{H_1I_2\dots I_k}]_q+C_{H_kI_k\dots H_2I_2H_1\, 1}C_{I_2\dots I_k}\nonumber\\
\hphantom{C_{H_kH_{k-1}\dots H_1\, 1}=}{} +C_{H_1}C_{H_kH_{k-1}\dots H_2 ,1} .\end{gather}
For $k=2$, this reads
\begin{equation}\label{app-r0-2}
C_{H_2H_1\, 1}=-[C_{H_2I_2\, 1},C_{H_1I_2}]_q+C_{H_2I_2H_1\, 1}C_{I_2}+C_{H_1}C_{H_2 ,1} .
\end{equation}
If $\{1\}$ is adjacent to $H_1$, this relation is a particular case of~\eqref{relaw44}. If $\{1\}$ is not adjacent to~$H_1$, the relation we know from~\eqref{relaw44} is~\eqref{app-r0-2} with $1$ replaced by $a$, the letter adjacent to~$H_1$. Then we can apply to that relation a sequence of automorphisms, namely, $\rb_1\dots \rb_{a-1}$, to get~\eqref{app-r0-2}.

Now from~\eqref{app-r0-2}, we can apply a sequence of automorphisms $\rb_i$ to split $I_2$ into a sequence $I'_2<\dots <I'_k$ with holes, and get~\eqref{app-r0-k} in the general case. To show this, let $H_1=\{a,\dots,b-1\}$ and $I_2=\{b,\dots,c\}$. Then relation~\eqref{app-r0-2} reads
\[C_{n\dots c+1 , b-1\dots a ,1}=-[C_{n\dots b ,1},C_{a\dots b \dots c}]_q+C_{n\dots a , 1}C_{b\dots c}+C_{a\dots b-1}C_{n\dots c+1 ,1} .\]
Applying $\rb_c$, according to Proposition~\ref{prop-act-ri}, we get
\[C_{n\dots c+2 , c , b-1\dots a ,1}=-[C_{n\dots b ,1},C_{a\dots b \dots c-1 ,c+1}]_q+C_{n\dots a , 1}C_{b\dots c-1 ,c+1}+C_{a\dots b-1}C_{n\dots c+2 ,c ,1} .\]
Thus we have successfully created $I'_2=\{b,\dots,c-1\}$, $I'_3=\{c+1\}$ and a hole $H_2=\{c\}$ such that $I_2=I'_2I'_3$. We can reproduce this as much as we need to produce as many holes as we want in $I_2$.

\subsection[Proof of the morphism property for r\_0]{Proof of the morphism property for $\boldsymbol{r_0}$} First, we show that $r_0$ preserves the commutation relations~\eqref{relcommv}. We start with the situation where $r_0$ acts non-trivially only on one element of the commutator, that is, with ${[C_{1\dots k},C_{i\dots j}]=0}$ with $1\neq i\leq j$. The image by $r_0$ is ${[C_{n\dots k+1\, 1},C_{i\dots j}]=0}$. This is satisfied by Lemma~\ref{lem:Commg} since if $\{i,\dots,j\}$ is included in (resp. disjoint from) $\{1,\dots,k\}$, then is it disjoint from (resp.\ included in) $\{n,\dots,k+1\}\cup\{1\}$.

Next, we consider the relation $[C_{1\dots k},\!C_{1\dots k'}]\!=\!0$, which is sent by $r_0$ to ${[C_{n\dots k+1\,1},\!C_{n\dots k'+1\,1}]\!=\!0}$. This latter relation is satisfied due to~\eqref{rel:coma3}.

Then we consider~\eqref{relaw31v} and~\eqref{relaw41v}. As for the proof for the other $r_i$'s, the situation where $r_0$ leaves invariant the two elements appearing in the $q$-commutator is easily handled, since $r_0$ leaves invariant every elements appearing in the relation. When $r_0$ leaves invariant only one element in the $q$-commutator, we can apply the same reasoning as around~\eqref{app-relKIJ}, using the formula in Proposition~\ref{prop-ri} and $q$-Jacobi. We omit the details since they follow the same reasoning as for the other $r_i$'s.

We are left with the situation where $r_0$ acts non-trivially on the two elements appearing in the $q$-commutator. This can happen only for relation~\eqref{relaw31v}, when $1\in I_3$ (so that $I_3<I_2<I_1$). Here we use Proposition~\ref{prop-act-ri} to apply $r_0$ and get
\[C_{I_1I_2}=-[C_{I_0I_2\,1},C_{I_0I_1\,1}]_q+C_{I_1}C_{I_2}+C_{I_0I_1I_2\,1}C_{I_0\,1} ,\]
where $I_0$ is the subset adjacent to $I_1$ and going up to $n$. If $\{1\}$ is adjacent to $I_2$, then this is relation~\eqref{eq:2h1}. The general case is obtained with the same reasoning as for~\eqref{app-r0-2}.

\subsection*{Computer programs}
We used formal calculation softwares (FORM and Maple) for some proofs done in this article: the corresponding programs are available upon request to the authors.

\subsection*{Acknowledgements} The authors thank A.~Lacabanne for fruitful discussions.
N.~Cramp\'e and L.~Poulain~d'Andecy thank LAPTh for its hospitality and are supported by the international research project AAPT of the CNRS and the ANR Project AHA ANR-18-CE40-000.
This work is partially supported by Universit\'e Savoie Mont Blanc and Conseil Savoie Mont Blanc grant APOINT.
We also wish to thank the anonymous referees for their valuable comments during the course of revision.

\pdfbookmark[1]{References}{ref}
\LastPageEnding

\end{document}